\documentclass[a4paper]{amsart}
\usepackage{url}
\usepackage{graphicx}
\usepackage[all]{xy}
\usepackage[mathscr]{eucal}
\usepackage{amsmath,amssymb,amsfonts}
\newtheorem{theorem}{Theorem}[section]
\newtheorem{lemma}[theorem]{Lemma}
\newtheorem{corollary}[theorem]{Corollary}
\newtheorem{example}[theorem]{Example}

\newtheorem{proposition}[theorem]{Proposition}
\newtheorem{remark}[theorem]{Remark}

\usepackage{stackengine}
\usepackage{xcolor}
\usepackage[all]{xy}

\title{Finite and semisimple Boolean inverse monoids}

\author{Mark V. Lawson}
\address{Mark V. Lawson, Department of Mathematics
and the
Maxwell Institute for Mathematical Sciences,
Heriot-Watt University,
Riccarton,
Edinburgh EH14 4AS,
UNITED KINGDOM}
\email{m.v.lawson@hw.ac.uk}

\thanks{I would like to thank David Janin of the Universit\'e de Bordeaux for the opportunity to visit him in April 2018 where much of the work for this paper was carried out. }
\begin{document}

%%%%%%%%%%%%%%%%%%%%%%%%%%%%%%%%%%%%%%%%%%%%%%%%%%%%%%%%%%%%%%%%%%%%%%%%%%%%%%%%%%%%%%%%%%%%%%
\begin{abstract}
We describe the structure of finite Boolean inverse monoids and apply our results to the representation theory of finite inverse semigroups.
We then generalize to semisimple Boolean inverse semigroups.
\end{abstract}
\maketitle

%%%%%%%%%%%%%%%%%%%%%%%%%%%%%%%%%%%%%%%%%%%%%%%%%%%%%%%%%%%%%%%%%%%%%%%%%%%%%%%%%%%%%%%%%%%%%%%%%
\section{Introduction}

The goal of this paper is to prove a number of theorems about finite and semisimple Boolean inverse semigroups.
Although semisimple Boolean inverse semigroups include finite Boolean inverse monoids as a special case,
we treat them separately since the finite case can be handled using only elementary means.

It turns out that the theory of finite Boolean inverse monoids, described in Section 4, is tightly intertwined with the theory of finite groupoids.
In fact, the theory of finite Boolean inverse monoids generalizes the theory of finite Boolean algebras
in exactly the same way as the theory of finite groupoids generalizes the theory of finite sets:
specifically, we prove in Theorem~\ref{them:main-finite} that every finite Boolean inverse monoid is isomorphic to the Boolean
inverse monoid of all local bisections of a finite discrete groupoid: namely,
the finite discrete groupoid of its atoms;
in Theorem~\ref{them:finite}, we prove that every finite Boolean inverse monoid
is isomorphic to a finite direct product of $0$-simplifying Boolean inverse monoids
each of which is isomorphic to all the finite square rook matrices over a finite group with a zero adjoined;
in Theorem~\ref{them:booleanization-finite}, we prove that the Booleanization of a finite inverse semigroup (without zero) $S$
is isomorphic to the Boolean inverse monoid of all local bisections of the set $S$ equipped with the restricted product.
We apply some of our results on the finite case to the study of the representation theory of finite inverse semigroups.

The theory of semisimple Boolean inverse semigroups, described in Section 5,
generalizes the theory of finite Boolean inverse monoids but we use non-commutative Stone duality to achieve our goals.
In Theorem~\ref{them:discrete-topology}, we prove that the semisimple Boolean inverse semigroups
are precisely those in which the associated groupoid of prime filters is endowed with the discrete topology;
this leads to our characterization of semisimple Boolean inverse semigroups in Corollary~\ref{cor:semisimple}
as being isomorphic to the Boolean inverse semigroup of all finite local bisections of a discrete groupoid;
in Theorem~\ref{them:struct-semisimple}, we prove that every semisimple Boolean inverse semigroup
is isomorphic to a {\em restricted} direct product of $0$-simplifying Boolean inverse semigroups each of which
is isomorphic to all the square rook matrices over a group with zero adjoined.
But it is in Theorem~\ref{them:dichotomy}, the Dichotomy Theorem, in which the r\^ole of semisimple Boolean inverse semigroups
within the wider structure theory of Boolean inverse semigroups is made clear:
we prove that every $0$-simplifying Boolean inverse semigroup is either semisimple or atomless.
A countable atomless Boolean inverse monoid is called a {\em Tarski monoid}.
The theory of such monoids is discussed in \cite{Lawson2016, Lawson2017}.
In Theorem~\ref{them:universal-groupoid}, we prove that an inverse semigroup is semisimple if and only if its universal groupoid is discrete. 

The type monoids of semisimple Boolean inverse semigroups are characterized in Section~6. 
Sections 2 and 3 contain background results we need.

%%%%%%%%%%%%%%%%%%%%%%%%%%%%%%%%%%%%%%%%%%%%%%%%%%%%%%%%%%%%%%%%%%%%%%%%%
\section{Background results}

%%%%%%%%%%%%%%%%%%%%%%%%%%%%%%%%%%%%%%%
\subsection{Posets}
Let $(X,\leq)$ be a poset.
If $A \subseteq X$ then define 
$$A^{\uparrow} = \{x \in X \colon a \leq x \text{ for some } a \in A\}$$
and 
$$A^{\downarrow} = \{x \in X \colon x \leq a \text{ for some } a \in A\}.$$
If $A = \{a\}$ then we write $a^{\uparrow} = \{a\}^{\uparrow}$
and  $a^{\downarrow} = \{a\}^{\downarrow}$.
If $A = A^{\downarrow}$ we say that $A$ is an {\em order-ideal}. 
If $A$ is a singleton set then $A^{\downarrow}$ is called a {\em principal order-ideal}.

%%%%%%%%%%%%%%%%%%%%%%%%%%%%%%%%%%%%%%%%%%%%%%%%%%%%%%%%%%%%%%%%%%%%%%%%
\subsection{Groupoids}
See the book by Higgins \cite{Higgins} for general groupoid theory.
For us a {\em groupoid} is a small category in which every element is invertible.
We denote the unique inverse of the element $x$ by $x^{-1}$.
If $G$ is a groupoid we write $\mathbf{d}(x) = x^{-1}x$, for the {\em domain} of $x$, and $\mathbf{r}(x) = xx^{-1}$, for the {\em range} of $x$.
Observe that for us categories are `one-sorted structures' and so we identify the objects of the
category with the identities.
The set of identities of $G$ is denoted by $G_{o}$.
If $e$ is an identity, define $G_{e}$ to be all elements $g$ such that $g^{-1}g = e = gg^{-1}$.
Then $G_{e}$ is a group called the {\em local group at $e$}.
%The union of all the local groups is a subgroupoid of $G$ called the {\em isotropy subgroupoid}.
%It is denoted by $\mbox{\rm Iso}(G)$.
A groupoid is said to be {\em principal} if all local groups are trivial.
A principal groupoid is just an equivalence relation viewed as a groupoid.
If $e$ and $f$ are identities we write $e \,\mathscr{D} \, f$ if there is an element of the groupoid
whose domain is $e$ and whose range is $f$.
If $g$ and $h$ are elements of the groupoid we write $g \,\mathscr{D} \, h$ if $\mathbf{d}(g) \,\mathscr{D} \, \mathbf{d}(h)$.
The relation $\mathscr{D}$ is an equivalence relation on $G$ whose equivalence classes are called the {\em connected components} of $G$.
%A subset $X \subseteq G$ of a groupoid is said to be {\em invariant} if it is a union of $\mathscr{D}$-classes.
A groupoid is said to be {\em connected} if it has exactly one connected component.
Every groupoid is a disjoint union of its connected components each of which is a connected groupoid.
The proof of the following is well-known.

\begin{lemma}\label{lem:connected-groupoids}
Let $H$ be a group and $X$ a non-empty set.
Then $X \times H \times X$ is a connected groupoid when the partial multiplication is defined by $(x,g,y)(y,h,z) = (x,gh,z)$.
Here, $\mathbf{d}(x,g,y) = (y,1,y)$, $\mathbf{r}(x,g,y) = (x,1,x)$ and $(x,g,y) = (y,g^{-1},x)$.
The identities are the elements $(x,1,x)$.
Every connected groupoid is isomorphic to a groupoid of this form.
\end{lemma}

Let $G$ be a groupoid.
A subset $A \subseteq G$ is called a {\em local bisection} if
$a,b \in A$ and $\mathbf{d}(a) = \mathbf{d}(b)$ (respectively, $\mathbf{r}(a) = \mathbf{r}(b)$) implies that $a = b$.
The set-theoretic product of two local bisections is a local bisection. 
Denote by $\mathsf{K}(G)$ the set of all local bisections of $G$.
Denote by $\mathsf{K}_{{\rm \tiny fin}} (G)$ the set of all finite local bisections of $G$.

%%%%%%%%%%%%%%%%%%%%%%%%%%%%%%%%%%%%%%%%%%%%%%%%%%%%%%%%%%%%%%%%%%%%
\subsection{Inverse semigroups}
For background results in inverse semigroup theory, we shall refer to \cite{Lawson1998}.
If $S$ is an inverse semigroup, we denote its semilattice of idempotents by $\mathsf{E}(S)$.
If $X \subseteq S$ then define $\mathsf{E}(X) = X \cap \mathsf{E}(S)$.
If $I \subseteq S$ then $I$ is a {\em (semigroup) ideal} if $SI, IS \subseteq I$.
One issue that will haunt this paper is the presence or absence of a zero
since it affects morphisms.
Thus homomorphisms between inverse semigroups with zero are required to preserve zero.
If $S$ is an inverse semigroup with a zero then $S^{\ast} = S \setminus \{0\}$.
If $S$ is an inverse semigroup without a zero then we can adjoin one to get $S^{0}$,
an inverse semigroup with zero.

If $\theta$ is a homomorphism with domain $S$ then the corresponding congruence induced on $S$  is denoted by $\mbox{cong}(\theta)$.

An inverse semigroup is said to be  {\em fundamental} if the only elements of the semigroup
commuting with all idempotents are themselves idempotents.
Equivalently, an inverse semigroup is fundamental if the maximum idempotent-separating congruence $\mu$ is trivial  \cite[Pages 139, 140]{Lawson1998}.
Fundamental inverse semigroups are of signal importance in the general theory of inverse semigroups;
see \cite[Section~5.2]{Lawson1998}.

If $a$ is an element of an inverse semigroup  $S$ then define $\mathbf{d}(a) = a^{-1}a$ and $\mathbf{r}(a) = aa^{-1}$;
we also write $\mathbf{d}(a) \stackrel{a}{\longrightarrow} \mathbf{r}(a)$ and $a \, \mathscr{D} \, b$.
The {\em restricted product} $a \cdot b$ of the elements $a$ and $b$ in an inverse semigroup is defined precisely when $\mathbf{d}(a) = \mathbf{r}(b)$
in which case it is equal to $ab$.
With respect to the restricted product, $S$ becomes a groupoid which we denote by $S^{\cdot}$ \cite[Proposition 3.1.4]{Lawson1998}.
Observe that  if $a$ and $b$ are both nonzero and $\mathbf{d}(a) = \mathbf{r}(b)$ then $a \cdot b$ is also nonzero.
The following therefore makes sense:
define $\mathcal{G}(S)$ to be the groupoid $S^{\cdot}$ if $S$ does not have a zero and the groupoid 
with set $S^{\ast}$ equipped with the restricted product if $S$ does have a zero.
We call $\mathcal{G}(S)$ the {\em restricted groupoid} of $S$.

The {\em natural partial order} is defined by $a \leq b$ if and only if $a = ba^{-1}a$.
An inverse semigroup is partially ordered with respect to the natural partial order.
An inverse semigroup with all binary meets with respect to the natural partial order is called a
{\em meet-semigroup} or {\em $\wedge$-semigroup}.
The proof of the following is straightforward from the properties of the natural partial order.

\begin{lemma}\label{lem:l-and-r-order} Let $S$ be an inverse semigroup.
If $x,y \leq a$ and $\mathbf{d}(x) = \mathbf {d}(y)$ (respectively, $\mathbf{r}(x) = \mathbf{r}(y)$)
then $x = y$.
\end{lemma}

An element $a$ of an inverse semigroup is called an {\em atom} if $x \leq a$ implies $x = a$ or $x = 0$.
An inverse semigroup is said to be {\em atomless} if it has no atoms. 

\begin{lemma}\label{lem:atom-idempotent} Let $S$ be an inverse semigroup
in which $a \mathscr{D} b$.
Then $a$ is an atom if and only if $b$ is an atom.
\end{lemma}
\begin{proof} We prove that  $a$ an atom implies that $b$ is an atom.
The result then follows by symmetry.
Let $\mathbf{d}(b) \stackrel{x}{\rightarrow} \mathbf{d}(a)$.
Suppose that $c \leq b$.
Then $a\mathbf{r}(x\mathbf{d}(c)) \leq a$.
But $a$ is an atom.
It follows that 
$a\mathbf{r}(x\mathbf{d}(c)) = 0$
or
$a\mathbf{r}(x\mathbf{d}(c)) = a$.
If the former, then $\mathbf{r}(x\mathbf{d}(x)) = 0$ and so $x\mathbf{d}(c) = 0$ which yields
$\mathbf{d}(c) = 0$ and so $c = 0$.
If the latter then $\mathbf{d}(a) = \mathbf{r}(x\mathbf{d}(c))$.
It follows that $\mathbf{d}(c) = \mathbf{d}(x) = \mathbf{d}(b)$.
Thus $c = b$.
We have accordingly proved that $b$ is an atom.
\end{proof}

The {\em compatibility relation} is defined by $a \sim b$ if and only if $a^{-1}b$ and $ab^{-1}$ are both idempotents;
the {\em orthogonality relation} is defined by $a \perp b$ if and only if $a^{-1}b = 0$ and $ab^{-1} = 0$.
A necessary condition for two elements of an inverse semigroup to have a join is that they be compatible.
A non-empty subset of an inverse semigroup is said to be {\em compatible} if each pair of elements of the subset is compatible. 
The following is \cite[Lemma 1.4.11, Lemma 1.4.12]{Lawson1998}.

\begin{lemma}\label{lem:wedge} Let $S$ be an inverse semigroup.
\begin{enumerate}

\item $a \sim b$ if and only if $a \wedge b$ exists and 
$\mathbf{d}(a \wedge b) = \mathbf{d}(a)\mathbf{d}(b)$ 
and
$\mathbf{r}(a \wedge b) = \mathbf{r}(a)\mathbf{r}(b)$.

\item If $a \sim b$ then $a \wedge b = ab^{-1}b = bb^{-1}a = aa^{-1}b = ba^{-1}a$.

\end{enumerate}
\end{lemma} 

The following was proved as \cite[Proposition~1.4.19]{Lawson1998}.

\begin{lemma}\label{lem:fish} Let $S$ be an inverse semigroup, let $A$ be a non-empty subset and let $s \in S$.
\begin{enumerate}
\item If $\bigwedge A$ exists then $\bigwedge_{a \in A}sa$ exists and $s\left( \bigwedge A \right) = \bigwedge_{a \in A}sa$. 
\item If $\bigwedge A$ exists then $\bigwedge_{a \in A}as$ exists and $\left( \bigwedge A \right)s = \bigwedge_{a \in A}as$. 
\end{enumerate}
\end{lemma}

The proof of the following is straightforward.

\begin{lemma}\label{lem:oj} Let $S$ be an inverse semigroup.
\begin{enumerate}
\item $a \perp b$ implies that $ca \perp cb$ for any $c \in S$.
\item $a \perp b$ implies that $ac \perp bc$ for any $c \in S$.
\end{enumerate}
\end{lemma}

\begin{lemma}\label{lem:buffs} In an inverse semigroup, let $a \sim b$.
Then the following are equivalent:
\begin{enumerate}
\item $a \perp b$.
\item $\mathbf{d}(a) \perp \mathbf{d}(b)$.
\item $\mathbf{r}(a) \perp \mathbf{r}(b)$.
\end{enumerate}
\end{lemma}
\begin{proof} By symmetry, it is enough to prove the equivalence of (1) and (2).
Clearly, (1) implies (2).
We prove that (2) implies (1).
Since $a \sim b$ we have that $a^{-1}b$ is an idempotent.
Hence $aa^{-1}bb^{-1} \leq ab^{-1}$.
But $ab^{-1} = a(a^{-1}ab^{-1}b)b^{-1} = 0$.
Thus  $aa^{-1}bb^{-1} \leq ab^{-1} = 0$.
We have proved that $\mathbf{r}(a) \perp \mathbf{r}(b)$ and so $a \perp b$.
\end{proof}

We shall need the following properties of the restricted product.

\begin{lemma}\label{lem:restricted-product} Let $S$ be an inverse semigroup.
\begin{enumerate}
\item Let $a,b \in S$. Then $ab = a' \cdot b'$ where $a' \leq a$ and $b' \leq b$.
\item Suppose that $ab$ is not a restricted product.
Then $ab = a_{1}b$, where $a_{1} < a$,
or $ab = ab_{1}$, where $b_{1} < b$.
\item Let $x \leq a \cdot b$.
Then $x = a' \cdot b'$ where $a' \leq a$ and $b' \leq b$.
\item Let $a,b \in S$. 
Then $(a^{\downarrow})(b^{\downarrow}) = (ab)^{\downarrow}$, where on the left we use products of subsets.
\end{enumerate}
\end{lemma}
\begin{proof} 
(1) Put $a' = abb^{-1}$ and $b' = a^{-1}ab$.
It is now routine to check that $\mathbf{d}(a') = \mathbf{r}(b)$
and that $ab = a' \cdot b'$.

(2) We are given that $a^{-1}a \neq bb^{-1}$.
Suppose that $abb^{-1} = a$ and $a^{-1}ab = b$.
Then $a^{-1}abb^{-1} = a^{-1}a$ and so $a^{-1}a \leq bb^{-1}$
and $a^{-1}abb^{-1} = bb^{-1}$ and so $bb^{-1} \leq a^{-1}a$.
From which it follows that $a^{-1}a = bb^{-1}$ which is a contradiction.
It follows that either $abb^{-1} \neq a$ or $a^{-1}ab \neq  b$.
The result is now clear since $abb^{-1} \leq a$ and $a^{-1}ab \leq b$.

(3) We have that $x = a(b\mathbf{d}(x))$.
Thus by (1) above, we have that $x = a' \cdot b'$ where
$a' \leq a$ and $b' \leq b\mathbf{d}(x) \leq b$.

(4) The fact that $(a^{\downarrow})(b^{\downarrow}) \subseteq (ab)^{\downarrow}$
is immediate from the fact that inverse semigroups are partially ordered with respect to the natural partial order.
The reverse inclusion follows by (3).
\end{proof}

%%%%%%%%%%%%%%%%%%%%%%%%%%%%%%%%%%%%%%%%%%%%%%%%%%%%%%%%%%%%%%%%%%%%%%%%%%%%%%%%
\section{Boolean inverse semigroups}

In this section, we shall describe the properties of Boolean inverse semigroups we shall need.
We refer the reader to \cite{Wehrung} for the general theory of Boolean inverse semigroups.

%%%%%%%%%%%%%%%%%%%%%%%%%%%%%%%%%%%%%%%%%%%%%%%%%%%%
\subsection{Lattices}
A {\em distributive lattice} has a bottom element but not necessarily a top element; if it does have a top, it is said to be {\em unital}.
For the theory of Boolean algebras, we refer to \cite{GH} but we highlight some non-standard terminology here.
We use the term {\em unital Boolean algebra} to mean what is usually defined to be simply a `Boolean algebra'
and  {\em Boolean algebra} to mean what is usually termed a `generalized Boolean algebra'.
Specifically, a Boolean algebra is a distributive lattice that is locally a unital Boolean algebra ---
this means that for each element $x$ in the lattice the subset $x^{\downarrow}$ of elements beneath $x$ is a unital Boolean algebra.
In a Boolean algebra, if $e \leq f$ then there is a unique element, denoted by $(f \setminus e)$, such that $f = e \vee (f \setminus e)$
and $e \wedge (f \setminus e) = 0$.
The element $(f \setminus e)$ is the {\em relative complement}.
Observe that $(f \setminus e)$ is the unique element $g$ such that  $g \leq f$ such that $g \wedge e = 0$ and $f = e \vee g$.
In a unital Boolean algebra, we denote the (absolute) complement of an element $e$ by $\bar{e}$;
as usual, we have that $1 = e \vee \bar{e}$ and $e \wedge \bar{e} = 0$.
Observe that
$e \leq f$ if and only if $e\bar{f} = 0$.
If $X$ is any set then $\mathsf{P}(X)$ denotes the {\em powerset} of $X$.
This is a unital Boolean algebra for subset inclusion.
If $X$ is an infinite set then $\mathsf{P}_{\rm \tiny fin}(X)$ denote the set of all finite
subsets of $X$.
This is a Boolean algebra.

%%%%%%%%%%%%%%%%%%%%%%%%%%%%%%%%%%%%%%%%%%%%%%%%%%
\subsection{Definition of Boolean inverse semigroups}
An inverse semigroup is said to be {\em distributive} if compatible elements have joins and multiplication distributes over such joins.
A distributive inverse semigroup is said to be {\em Boolean} if the semilattice of idempotents forms a Boolean algebra.

\begin{remark}
{\em We shall assume throughout this paper that our Boolean inverse semigroups are not just the zero semigroup.}
\end{remark}

If $a \perp b$ then $a \vee b$ is often written as $a \oplus b$ and is called an {\em orthogonal join}.
The following class of examples of Boolean inverse semigroups is fundamental to this paper.

\begin{proposition}\label{prop:boolean-finite-local-bisections} Let $G$ be a discrete groupoid.
Then
$\mathsf{K}_{{\rm \tiny fin}} (G)$, the set of all finite local bisections of $G$,
is a Boolean inverse semigroup.
\end{proposition}
\begin{proof} The set of all local bisections of $G$ is a pseudogroup \cite[Proposition~2.1]{LL}. 
The product of two finite local bisections is a finite local bisection.
Observe that the inverse of $A$ is $A^{-1}$.
It follows that $\mathsf{K}_{{\rm \tiny fin}} (G)$ is a distributive inverse semigroup.
The finite local bisections in the set of identities of $G$ are the finite subsets of the set of identities.
Thus the poset of idempotents of $\mathsf{K}_{{\rm \tiny fin}} (G)$ is order-isomorphic to the set
of all finite subsets of the set of identities of $G$.
Thus $\mathsf{K}_{{\rm \tiny fin}} (G)$ is a Boolean inverse semigroup.
\end{proof}

The following  two results were proved as \cite[Lemma~2.5(3)]{Lawson2016}
and \cite[Lemma~2.5(4)]{Lawson2016}, respectively.

\begin{lemma}\label{lem:eggs} Let $S$ be a distributive inverse semigroup.
\begin{enumerate}

\item Suppose that $\bigvee_{i=1}^{m}a_{i}$ and $c \wedge \left(  \bigvee_{i=1}^{m}a_{i}  \right)$ both exist.
Then all meets $c \wedge a_{i}$ exist as does the join  $\bigvee_{i=1}^{m}(a_{i} \wedge c)$ and we have that
$$c \wedge \left(  \bigvee_{i=1}^{m}a_{i}  \right)
=
\bigvee_{i=1}^{m}(a_{i} \wedge c).$$

\item Suppose that $a \in S$ and $b = \bigvee_{j=1}^{n} b_{j}$ are such that all meets $a \wedge b_{j}$ exist.
Then $a \wedge b$ exists and is equal to $\bigvee_{j=1}^{n} a \wedge b_{j}$.

\end{enumerate}
\end{lemma}

%If $S$ is a Boolean inverse semigroup and $a \in S$, 
%define $\mathbf{e}(a) = \mathbf{d}(a) \vee \mathbf{r}(a)$
%and call this idempotent the {\em extent} of $a$.
If $y \leq x$ then $\mathbf{d}(y) \leq \mathbf{d}(x)$ and so the following definition makes sense:
$$(x \setminus y) = x(\mathbf{d}(x) \setminus \mathbf{d}(y)).$$

\begin{lemma}\label{lem:chicken} Let $S$ be a Boolean inverse semigroup.
\begin{enumerate}
\item $\mathbf{d}(x \setminus y) = (\mathbf{d}(x) \setminus \mathbf{d}(y))$.
\item If $y \leq x$ then $y \perp (x \setminus y)$ and $x = y \vee (x \setminus y)$.
\item Suppose that $a \leq x$ is such that $a \perp y$ and $x = y \vee a$ then
$a = (x \setminus y)$. 
\item $\mathbf{r}(x \setminus y) = (\mathbf{r}(x) \setminus \mathbf{r}(y))$.
\end{enumerate}
\end{lemma}
\begin{proof} (1) Straightforward.

(2) Observe that $(x \setminus y) \leq x$
and that
$\mathbf{d}(x \setminus y) = \mathbf{d}(x) \setminus \mathbf{d}(y)$.
It follows that $x \sim (x \setminus y)$ and so, by Lemma~\ref{lem:buffs},
we have that $y \perp (x \setminus y)$.
Clearly, $y \vee (x \setminus y) \leq x$.
But $\mathbf{d}(y \vee (x \setminus y)) = \mathbf{d}(y) \vee (\mathbf{d}(x) \setminus \mathbf{d}(y)) = \mathbf{d}(x)$.
It follows that $x = y \vee (x \setminus y) \leq x$.

(3) Observe that $\mathbf{d}(a) \leq \mathbf{d}(x)$, $\mathbf{d}(a) \perp \mathbf{d}(y)$
and $\mathbf{d}(x) = \mathbf{d}(y) \vee \mathbf{d}(a)$.
Thus in the Boolean algebra, we have that $\mathbf{d}(a) \leq \mathbf{d}(x) \setminus \mathbf{d}(y)$.
It follows that $a \leq (x \setminus y)$.

(4) This follows by part (3) above.
\end{proof}

The following was proved as \cite[Lemma 2.2]{Lawson2020} and \cite[Lemma~3.1.12]{Wehrung}.

\begin{lemma}\label{lem:properties-of-setminus} Let $S$ be a Boolean inverse semigroup.
\begin{enumerate}

\item $(s \setminus t) = s^{-1} \setminus t^{-1}$.

\item $a(s \setminus t) = as \setminus at$.
and
$(s \setminus t)a = sa \setminus ta$.

\item $s(u \vee v) = (s \setminus u)s^{-1}(s \setminus v)$.

\item $(s \setminus t)(u \setminus v) = su \setminus (sv \vee tu)$.

\item $a \leq b \leq c$ implies that $c \setminus c \leq c \setminus a$. 

\end{enumerate}
\end{lemma}

%\begin{lemma}\label{lem:tea} Let $S$ be a Boolean inverse semigroup and let $x,y \in S$ be arbitrary
%such that $x \wedge y$ is defined.
%Then $x = y \vee (x \setminus x \wedge y)$.
%\end{lemma}
%\begin{proof} We have that
%$y \vee (x \setminus x \wedge y) = (y \vee (x \wedge y)) \vee (x \setminus x \wedge y)$
%since $x \wedge y \leq y$.
%Now use the associativity of $\vee$ to get the result.
%\end{proof}

\begin{lemma}\label{lem:pork} Let $S$ be a Boolean inverse semigroup.
Then every binary join is equal to an orthogonal join.
\end{lemma} 
\begin{proof} Let $x \sim y$ so that $x \vee y$ is defined.
By Lemma~\ref{lem:wedge}, the meet $x \wedge y$ exists 
and
$\mathbf{d}(x \wedge y) = \mathbf{d}(x)\mathbf{d}(y)$ 
and
$\mathbf{r}(x \wedge y) = \mathbf{r}(x)\mathbf{r}(y)$.
Since $x \wedge y \leq x$ we can construct the element $x \setminus (x \wedge y)$.
Thus $x \vee y = (x \wedge y) \vee (x \setminus (x \wedge y)) \vee y$.
But $x \wedge y \leq y$.
It follows that 
$x \vee y = (x \setminus (x \wedge y)) \vee y$.
Clearly, $y \sim (x \setminus (x \wedge y))$ and $\mathbf{d}(y) \perp \mathbf{d}(x \setminus (x \wedge y))$.
It follows by Lemma~\ref{lem:buffs} that $y \perp x \setminus (x \wedge y)$.
\end{proof}

Although we have only talked above about compatible {\em binary} joins and orthogonal {\em binary} joins,
our results extend to $n$-ary joins of either type where $n$ is any non-zero natural number.

\begin{lemma}\label{lem:orthogonal} Let $s = \bigvee_{i=1}^{m} s_{i}$ where the $s_{i}$ are distinct and non-zero.
Then there is an orthogonal set of elements $\{t_{1}, \ldots, t_{m}\}$ such that
\begin{enumerate}
\item $s = \bigvee_{i=1}^{m} t_{i}$.
\item For each $i$, we have that $t_{i} \leq s_{i}$.
\end{enumerate}
\end{lemma}
\begin{proof} The set $\{s_{1}, \ldots, s_{m}\}$ is a compatible one.
Thus all the joins $s_{1}$, $s_{1} \vee s_{2}$, $s_{1} \vee s_{2} \vee s_{3}$, \ldots exist.
In addition, by Lemma~\ref{lem:wedge}, all the meets 
$s_{1} \wedge s_{2}$, $(s_{1} \vee s_{2}) \wedge s_{3}$, $(s_{1} \vee s_{2} \vee s_{3}) \wedge s_{4}$, \ldots exist.
Define $t_{1} = s_{1}$ and
$$t_{i} = s_{i} \setminus ((s_{1} \vee \ldots s_{i-1}) \wedge s_{i})$$
for each $i = 2, \ldots, m$.
Observe that $\{t_{1}, \ldots, t_{m}\}$ is also a compatible subset.
Now, 
$\mathbf{d}(t_{1}) = \mathbf{d}(s_{1})$ 
and for $i \geq 2$ we have that 
$\mathbf{d}(t_{i}) = \mathbf{d}(s_{i}) \setminus \mathbf{d}((s_{1} \vee \ldots \vee s_{i-1}) \wedge s_{i})$.
In particular, $\mathbf{d}(t_{i}) \leq \mathbf{d}(s_{i})$.
Using Lemma~\ref{lem:buffs},
we have that 
$\mathbf{d}(t_{i}) \perp \mathbf{d}(s_{1}), \ldots, \mathbf{d}(s_{i-1})$.
It is therefore clear that $t_{1}$ is orthogonal to $t_{i}$ when $i \geq 2$
and that $t_{i}$ is orthogonal to $t_{j}$ when $i < j$.
Observe that $s_{i} = t_{i} \oplus ((s_{1} \vee \ldots \vee s_{i-1}) \wedge s_{i})$
Clearly, $\bigvee_{i=1}^{m} t_{i} \leq s$.
Now observe that $t_{1} = s_{1}$,
$t_{1} \vee t_{2} = s_{1} \vee s_{2}$ by Lemma~\ref{lem:pork}.
Assume that $t_{1} \vee \ldots \vee t_{i-1} = s_{1} \vee \ldots \vee s_{i-1}$.
Then 
$$t_{1} \vee \ldots \vee t_{i} 
= (s_{1} \vee \ldots \vee s_{i-1}) \vee t_{i}
= (s_{1} \vee \ldots \vee s_{i-1}) \vee (s_{i} \setminus ((s_{1} \vee \ldots s_{i-1}) \wedge s_{i})))$$
which is equal to $s_{1} \vee \ldots \vee s_{i}$.
\end{proof}

The following result is frequently invoked in proofs. 

\begin{proposition}\label{prop:definition} The following are equivalent.
\begin{enumerate}
\item $S$ is a Boolean inverse semigroup.
\item $S$ has all binary orthogonal joins, multiplication distributes over binary orthogonal joins, 
and its semilattice of idempotents forms a Boolean algebra with respect to the natural partial order.
\end{enumerate}
\end{proposition}
\begin{proof} It is enough to prove that (2)$\Rightarrow$(1).
We prove first that $S$ has all binary compatible joins.
Let $a,b \in S$ such that $a \sim b$.
Then, by Lemma~\ref{lem:wedge}, we know that $a \wedge b$ exists.
The idempotent $\mathbf{d}(a) \setminus \mathbf{d}(a)\mathbf{d}(b)$ is defined since $\mathsf{E}(S)$ is a Boolean algebra.
Define $a' = a(\mathbf{d}(a) \setminus \mathbf{d}(a)\mathbf{d}(b))$.
Then $a'$ and $a \wedge b$ are both less than $a$ and compatible.
In addition, $\mathbf{d}(a') \perp \mathbf{d}(a \wedge b)$.
Thus by Lemma~\ref{lem:buffs}, we have that $a' \oplus (a \wedge b)$ is defined.
It is clear that $a = a' \oplus (a \wedge b)$.
Similarly, $b ' = b(\mathbf{d}(b) \setminus \mathbf{d}(a) \mathbf{d}(b))$ is less than or equal to $b$
and $b = b' \oplus (a \wedge b)$.
The elements $a', (a \wedge b), b'$ are pairwise othogonal.
We have that $a \vee b = (a' \oplus (a \wedge b)) \oplus b'$.
Thus $a \vee b$ is defined.
Now let $c$ be any element.
We prove that $c(a \vee b) = ca \vee cb$.
We use Lemma~\ref{lem:oj}, so that if $a \perp b$ then $ca \perp cb$ and 
$c(a \oplus b) = ca \oplus cb$ by assumption.
Observe that
$c((a' \oplus (a \wedge b)) \oplus b') =  (ca' \oplus c(a \wedge b)) \oplus cb'$, by assumption.
But $c(a \wedge b) = ca \wedge cb$ by Lemma~\ref{lem:fish}.
Thus
$ca' = ca (\mathbf{d}(ca) \setminus \mathbf{d}(ca)\mathbf{d}(cb))$
and
$cb' =  cb(\mathbf{d}(cb) \setminus \mathbf{d}(ca) \mathbf{d}(cb))$ 
where we use the fact that multiplication distributes over orthogonal joins.
The result now follows.
\end{proof}

%BLOOP
%%%%%%%%%%%%%%%%%%%%%%%%%%%%%%%%%%%%%%%%%%%%%%%%%%%%%%%%%
An {\em additive homomorphism} $\theta \colon S \rightarrow T$ between Boolean inverse semigroups is a semigroup homomorphism
that maps zero to zero and preserves binary compatible joins.
Observe that an additive homomorphism $\theta$ induces a map from $\mathsf{E}(S)$ to $\mathsf{E}(T)$
that preserves meets and joins; in addition, if $f \leq e$ then $\theta (e \setminus f) = \theta (e) \setminus \theta (f)$.

If $S$ and $T$ are also both monoids then we say that $\theta$ is {\em unital} if it also maps the identity to the identity.
In what follows, we shall usually just write {\em morphism} rather than {\em additive homomorphism} for brevity.

%%%%%%%%%%%%%%%%%%%%%%%%%%%%%%%%%%%%%%%%%%%%%%%%
\subsection{Additive ideals}
Let $S$ be a Boolean inverse semigroup.
If $A \subseteq S$ define $A^{\vee}$ to be the set of all binary compatible joins of elements of $A$.
An {\em additive ideal} $I$ in a Boolean inverse semigroup is a semigroup ideal $I$ such that if $a,b \in I$ and
$a \sim b$ then $a \vee b \in I$.

\begin{lemma}\label{lem:smallest} Let $S$ be a Boolean inverse semigroup.
Let $A \subseteq S$ be a non-empty subset.
Then the smallest additive ideal containing $A$ is $I = (SAS)^{\vee}$.
\end{lemma} 
\begin{proof} Clearly, $A \subseteq I$.
We check first that $I$ is an additive ideal.
Let $a \in I$ and $s \in S$.
Then $a = \bigvee_{i=1}^{m} a_{i}$ where $a_{i} \in SAS$.
We have that $sa = \bigvee_{i=1}^{m} sa_{i}$.
But, clearly, $sa_{i} \in SAS$.
It follows that $sa \in I$.
A symmetric argument shows that $sa \in I$.
We have therefore proved that $I$ is a (semigroup) ideal.
Let $\{a_{1}, \ldots, a_{m}\}$ be a compatible subset of $I$.
Then, each $a_{i}$ is a join of elements of $SAS$.
Thus $\bigvee_{i=1}^{m} a_{i}$ is a join of elements of $SAS$.
It follows that $I$ is an additive ideal.
Now let $J$ be any additive ideal containing $A$.
Then since $A \subseteq J$ we have that $SAS \subseteq J$.
But $J$ is additive and so $(SAS)^{\vee} \subseteq J$.
We have therefore proved that $I \subseteq J$.
\end{proof}

Both $\{0\}$ and $S$ are additive ideals.
If these are the only additive ideals and $S \neq \{0\}$ then we say that $S$ is {\em $0$-simplifying}.
The following relation was introduced in \cite{Lenz}.
Let $e$ and $f$ be two non-zero idempotents in $S$.
Define $e \preceq f$ if and only if there exists a set of  elements $X = \{x_{1}, \ldots, x_{m}\}$ such that
$e = \bigvee_{i=1}^{m}  \mathbf{d}(x_{i})$
and 
$\mathbf{r}(x_{i}) \leq f$ for $1 \leq i \leq m$. 
We can write this informally as $e = \bigvee \mathbf{d}(X)$ and $\bigvee \mathbf{r}(X) \leq f$.
We say that $X$ is a {\em pencil} from $e$ to $f$.
The relation $\preceq$ is a preorder on the set of idempotents.

\begin{lemma}\label{lem:additive-ideals} Let $S$ be a Boolean inverse semigroup.
Then $f \preceq e$ if and only if whenever $e \in I$, an additive ideal,
then $f \in I$.
\end{lemma}
\begin{proof} Suppose that $f \preceq e$ and $e \in I$, an additive ideal.
Then there is a pencil $X$ from $f$ to $e$.
Thus $e = \bigvee_{x \in X}  \mathbf{d}(x)$
and 
$\mathbf{r}(x) \leq f$ for $1 \leq i \leq m$. 
Now $f \in I$ implies that $\mathbf{r}(x) \in I$ since $I$ is an order-ideal.
It follows that $\mathbf{d}(x) \in I$ since $I$ is a semigroup ideal.
But $I$ is closed under compatible joins and so $f \in I$, as required.

We now prove the converse.
The smallest additive ideal containing $e$ is $(SeS)^{\vee}$ by Lemma~\ref{lem:smallest}.
By assumption $f \in (SeS)^{\vee}$.
Thus $f = \bigvee_{i=1}^{m} e_{i}$ where $e_{1}, \ldots , e_{m} \in SeS$.
But it is easy to prove that $e_{i} = x_{i}^{-1}x_{i}$, where $x_{i} \in eSe_{i}$ 
Thus $x_{i}x_{i}^{-1} \leq e$.
It follows that $X = \{x_{1},\ldots,x_{m}\}$ is a pencil from $f$ to $e$
and so $f \preceq e$.
\end{proof}

Define the equivalence relation $e \equiv f$ if and only if $e \preceq f$ and $f \preceq e$.
The following was proved as part of \cite[Lemma~7.8]{Lenz} but is also immediate by Lemma~\ref{lem:additive-ideals}.

\begin{lemma}\label{lem:toby} Let $S$ be a Boolean inverse semigroup.
Then $\equiv$ is the universal relation on the set of non-zero idempotents if and only if $S$ is $0$-simplifying.
\end{lemma}

The {\em kernel} of a morphism $\theta \colon S \rightarrow T$ between two Boolean inverse semigroups, 
denoted by $\mbox{ker}(\theta)$, is the inverse image under $\theta$ of the zero of $T$.

\begin{remark}
{\em Readers familiar with semigroup theory are warned that our definition of `kernel' is not the one usual in semigroup theory.} 
\end{remark}

The proof of the following is straightforward.

\begin{lemma}\label{lem:additive-ideals-new} The kernels of morphisms
between Boolean inverse semigroups are additive ideals.
\end{lemma}

Lemma~\ref{lem:additive-ideals-new} immediately raises the question of what can be said about morphisms whose kernels are trivial.

\begin{lemma}\label{lem:idept-sep-kernel} Let $\theta \colon S \rightarrow T$ be a morphism between Boolean inverse semigroups.
Then $\theta$ has a trivial kernel if and only if it is idempotent-separating.
\end{lemma}
\begin{proof} Suppose first that $\theta$ is idempotent-separating.
Let $\theta (a) = 0$.
Then $\theta (a^{-1}a) = 0$.
It follows by assumption that $a^{-1}a = 0$ and so $a = 0$
which implies that the $\theta$ is trivial.
Conversely, suppose that the kernel of $\theta$ is trivial.
We prove that $\theta$ is idempotent-separating.
Let $\theta (e) = \theta (f)$ where $e$ and $f$ are idempotents.
Then $\theta (e) = \theta (e \wedge f) = \theta (f)$
since $\theta$ restricts to a morphism of Boolean algebras.
But then $\theta (e \setminus (e \wedge f)) = 0$ implies that $e = e \wedge f$.
Similarly $f = e \wedge f$.
Thus $e = f$, as required.
\end{proof}

%%%%%%%%%%%%%%%%%%%%%%%%%%%%%%%%%%%%%%%%%%%%%%%%%%%%%%%%%%%%
%CONGRUENCES
\subsection{Additive congruences}

The goal of this subsection is to define what we mean by a `simple' Boolean inverse semigroup.

We say that a congruence $\sigma$ on a Boolean inverse semigroup $S$ is {\em additive}
if $S/\sigma$ is a Boolean inverse semigroup and the natural map from $S$ to $S/\sigma$ is a morphism.
The treatment of additive congruences in \cite{Wehrung} cannot be bettered,
and I shall simply summarize the main definitions and results from there below.
It is easy to check that a congruence is additive precisely when it preserves the operations  $\ominus$ and $\triangledown$
introduced in \cite[Page 82]{Wehrung}.
The following is \cite[Proposition 3-4.1]{Wehrung}.

\begin{proposition}\label{prop:boolean-congruence} A congruence $\sigma$ on a Boolean inverse semigroup
$S$ is additive if and only if for all $a \in S$ and orthogonal idempotents $e$ and $f$ we have that
$ae \, \sigma \, e$ and $af \, \sigma \, f$ implies that $a(e \vee f) \, \sigma \, (e \vee f)$.
\end{proposition}

Observe that \cite[Example~3-4.2]{Wehrung} shows that not all idempotent-separating homomorphisms between Boolean inverse semigroups need be morphisms.
The following was proved as \cite[Proposition 3-4.5]{Wehrung}.

\begin{proposition}\label{lem:idem-sep} Let $S$ be a Boolean inverse semigroup.
Then $S/\mu$ is a Boolean inverse semigroup and the natural map $S \rightarrow S/\mu$ is a morphism of Boolean inverse semigroups.
\end{proposition}

The following is \cite[Proposition 3-4.6]{Wehrung}.

\begin{proposition}\label{prop:noise} Let $S$ be a Boolean inverse semigroup.
Let $I$ be an additive ideal of $S$.
\begin{enumerate}
\item Define $(a,b) \in \varepsilon_{I}$ if and only if there exists $c \leq a,b$ such that $a \setminus c, b \setminus c \in I$.
Then $\varepsilon_{I}$ is an additive congruence with kernel $I$.
\item If $\sigma$ is any additive congruence with kernel $I$ then $ \varepsilon_{I} \subseteq \sigma$.
\end{enumerate}
\end{proposition}

An additive congruence is {\em ideal-induced} if it equals $\varepsilon_{I}$ for some additive ideal $I$.
Let $S$ and $T$ be Boolean inverse semigroups.
A morphism $\theta \colon S \rightarrow T$ is said to be {\em weakly-meet-preserving} if
for any $a,b  \in S$ and any $t \in T$  if $t \leq \theta (a), \theta (b)$ then there exists $c \leq a,b$ such that $t \leq \theta (c)$.
Such morphisms were introduced in \cite{LL}. 
The following result is due to Ganna Kudryavtseva (private communication).

\begin{proposition}\label{prop:anja} A morphism of Boolean inverse semigroups is weakly-meet-pre\-serving if
and only if its associated congruence is ideal-induced.
\end{proposition}
\begin{proof} Let $I$ be an additive ideal of $S$ and let $\varepsilon_{I}$ be its associated additive congruence on $S$.
Denote by $\nu \colon S \rightarrow S/\varepsilon_{I}$ is associated natural morphism.
We prove that $\nu$ is weakly meet preserving.
Denote the $\varepsilon_{I}$-class containing $s$ by $[s]$.
Let $[t] \leq [a], [b]$.
Then $[t] = [at^{-1}t]$ and $[t] = [bt^{-1}t]$.
By definition there exist $u,v \in S$ such that
$u \leq t,at^{-1}t$ and $v \leq t,bt^{-1}t$ such that
$t \setminus u, at^{-1}t \setminus u, t \setminus v, bt^{-1}t \setminus v \in I$.
Now $[t] = [u] = [at^{-1}t]$ and $[t] = [v] = [bt^{-1}t]$.
Since $u,v \leq t$ it follows that $u \sim v$ and so $u \wedge v$ exists by Lemma~\ref{lem:wedge}.
Clearly, $u \wedge v \leq a,b$.
In addition $[t] = [u \wedge v]$.
We have proved that $\nu$ is weakly-meet-preserving.

Conversely, let $\theta \colon S \rightarrow T$ be weakly-meet-preserving.
We prove that it is determined by its kernel $I$.
By part (2) of Proposition~\ref{prop:noise}, it is enough to prove that if $\theta (a) = \theta (b)$ then
we can find $c \leq a,b$ such that $a \setminus c, b \setminus c \in I$.
Put $t = \theta (a) = \theta (b)$.
Then there exists $c \leq a,b$ such that $t \leq \theta (c)$.
It is easy to check that $\theta (a \setminus c) = 0 = \theta (b \setminus c)$.
We have therefore proved that $a \setminus c, b \setminus c \in I$ and so $(a,b) \in \varepsilon_{I}$.
\end{proof}

Most of the following is \cite[Proposition 3-4.9]{Wehrung}.

\begin{proposition}[Factorization of Boolean morphisms]\label{prop:factorization} Let $\theta \colon S \rightarrow T$ be a morphism of Boolean inverse semigroups.
Let the kernel of $\theta$ be $I$.
Put $\varepsilon = \varepsilon_{I}$ and let $\nu \colon S \rightarrow S/\varepsilon$ be the natural map.
Then there is a unique morphism $\phi \colon S/\varepsilon \rightarrow T$ such that
$\phi \nu = \theta$ where $\phi$ is idempotent-separating.
\end{proposition}
\begin{proof} Observe first that $\varepsilon \subseteq \mbox{\rm cong}(\theta)$;
suppose that $(a,b) \in \varepsilon$.
Then there exists $c \leq a,b$ such that $a \setminus c, b \setminus c \in I$.
We have that, $\theta (a) = \theta ((a\setminus c) \vee c) = \theta (c)$ and, by symmetry, $\theta (b) = \theta (c)$.
It follows that $(a,b) \in \mbox{\rm cong}(\theta)$.
The function $\phi \colon S/\varepsilon \rightarrow T$ defined by $\phi ([a]) = \theta (a)$,
where $[a]$ denotes the $\varepsilon$-class of $a$, is a well-defined homomorphism.
We prove that $\phi$ is idempotent-separating.
Let $\phi ([e]) = \phi ([f])$ where $e$ and $f$ are idempotents.
Then $\theta (e) = \theta (f)$.
Thus $\theta (e \setminus ef) = 0$ and $\theta (f \setminus ef) = 0$.
We have shown that $(e,f) \in \varepsilon$
and so $[e] = [f]$.
\end{proof}

Proposition~\ref{prop:factorization} is included for interest.

A Boolean inverse semigroup is {\em simple} if it has no non-trivial additive congruences.

\begin{proposition}\label{prop:boolean-congruence-free} 
A Boolean inverse semigroup is simple if and only if
it is $0$-simplifying and fundamental.
\end{proposition}
\begin{proof} Let $S$ be a $0$-simplifying fundamental Boolean inverse semigroup.
Let $\sigma$ be any Boolean congruence on $S$.
If $\sigma (0) = S$ then $\sigma$ is the universal congruence.
If $\sigma \neq S$ then, since $S$ is $0$-simplifying,  $\sigma (0) = \{0\}$.
Thus by Lemma~\ref{lem:idept-sep-kernel}, the congruence $\sigma$ is idempotent-separating.
It follows that $\sigma \subseteq \mu$.
But $S$ is fundamental and so $\mu$ is the equality congruence.
It follows that $\sigma$ is equality.
We have proved that $S$ is additive congruence-free.
Conversely, suppose now that $S$ is additive congruence-free.
Let $I$ be a $\vee$-ideal such that $I \neq \{0\}$.  
Then by Proposition~\ref{prop:noise}, we have that $\varepsilon_{I}$ is an additive congruence with kernel $I$.
It follows that  $\varepsilon_{I}$ is the universal congruence and so $I = S$.
We have proved that $S$ is $0$-simplifying.
By Lemma~\ref{lem:idem-sep}, the congruence $\mu$ is Boolean.
Thus by assumption $\mu$ is equality and so $S$ is fundamental.
\end{proof}

Proposition~\ref{prop:boolean-congruence-free} explains why in what follows
we are so interested in determining when a Boolean inverse semigroup is $0$-simplifying or fundamental.

%%%%%%%%%%%%%%%%%%%%%%%%%%%%%%%%%%%%%%%%%%%%%%%
\subsection{Generalized rook matrices}
Boolean inverse semigroups are more `ring-like' than arbitrary inverse semigroups.
One way this is demonstrated is that we may define matrices over such semigroups.
The following construction was first described in \cite{Hines} and then generalized in \cite{KLLR2016} and \cite{Wallis}.
Let $S$ be an arbitrary Boolean inverse semigroup and let $X$ be a non-empty set.
An {\em $|X| \times |X|$ generalized rook matrix over $S$} is an $|X| \times |X|$ matrix with entries from $S$ that satisfies the following three conditions:
\begin{description}

\item[{\rm (RM1)}] If $a$ and $b$ are in distinct columns and lie in the same row of $A$ then $a^{-1}b = 0$. That is $\mathbf{r}(a) \perp \mathbf{r}(b)$.

\item[{\rm (RM2)}] If $a$ and $b$ are in distinct rows and lie in the same column of $A$ then $ab^{-1} = 0$. That is $\mathbf{d}(a) \perp \mathbf{d}(b)$.

\item[{\rm (RM3)}] In the case where $|X|$ is infinite we also require that only a finite number of entries in the matrix are non-zero.

\end{description}
We shall usually just say {\em rook matrix} instead of {\em generalized rook matrix}.
When $n$ is finite, we use the notation $I_{n}$ to mean the $n \times n$ identity matrix.
Let $A$ and $B$ be rook matrices of the same size.
Then the matrix $AB$ is defined as follows:
$$(AB)_{ij} = \bigoplus_{k} a_{ik}b_{kj}.$$
The following is \cite[Proposition~3.3]{KLLR2016}.

\begin{proposition}\label{prop:ale} Let $S$ be a Boolean inverse semigroup. 
\begin{enumerate}

\item If $A$ and $B$ are rook matrices such that $AB$ is defined then $AB$ is well-defined and is a rook matrix.

\item Multiplication is associative when defined.

\item The matrices $I_{n}$ are identities when multiplication is defined.

\item Let $A = (a_{ij})$ be a rook matrix.
Define $A^{\ast} = (a_{ji}^{-1})$.
Then $A^{\ast}$ is a rook matrix and $A = AA^{\ast}A$ and $A^{\ast} = A^{\ast}AA^{\ast}$.

\item The idempotents are those square rook matrices which are diagonal and whose diagonal entries are themselves idempotents.

\item If $A$ and $B$ are two rook matrices of the same size then $A \leq B$ if and only if $a_{ij} \leq b_{ij}$ for all $i$ and $j$.

\item If $A$ and $B$ are again of the same size then $A \perp B$ if and only if $a_{ij} \perp b_{ij}$;
if this is the case then $A \vee B$ exists and its elements are $a_{ij} \vee b_{ij}$.

\end{enumerate}
\end{proposition}

Proposition~\ref{prop:ale} tells us, in particular, how to manipulate rook matrices over Boolean inverse semigroups.
The set of all $|X| \times |X|$ rook matrices over $S$ is denoted by $M_{|X|}(S)$.
If $X$ is finite then $|X| = n$, 
and we denote the corresponding set of rook matrices by  $M_{n}(S)$,
and if $X$ is countably infinite then  $|X| = \omega$,
and we denote the corresponding set of rook matrices by  $M_{\omega}(S)$.
The set $M_{|X|}(S)$ is a Boolean inverse semigroup.

\begin{example}\label{ex:solomon}
{\em Let $\mathbf{2} = \{0,1\}$, the two-element unital Boolean algebra.
Then $M_{n}(\mathbf{2})$ is the Boolean inverse monoid of $n \times n$ rook matrices in the sense of Solomon \cite{Solomon}, 
and so is isomorphic to the finite symmetric inverse monoid on $n$ letters $\mathscr{I}_{n}$.}
\end{example}

If $G$ is a group, then $G^{0}$ is the group with a zero adjoined.
This is a Boolean inverse monoid we call a {\em $0$-group.}
We proved in Proposition~\ref{prop:boolean-finite-local-bisections} that $\mathsf{K}_{\rm \tiny fin}(G)$ is a Boolean inverse semigroup 
for any discrete groupoid $G$.
In the case where the groupoid $G$ is connected, we can say more.

\begin{proposition}\label{prop:local-bisections-rook} Let $\mathscr{G}$ be a connected groupoid.
Then $\mathscr{G} = X \times G \times X$,
where $G$ is a group and $X$ is a non-empty set,
and the Boolean inverse semigroup $\mathsf{K}_{{\rm \tiny fin}} (\mathscr{G})$ is isomorphic to $M_{|X|}(G^{0})$.
\end{proposition}
\begin{proof} We first apply the structure theorem for connected groupoids Lemma~\ref{lem:connected-groupoids}.
We define an isomorphism 
$$\theta \colon \mathsf{K}_{{\rm \tiny fin}} (\mathscr{G}) \rightarrow M_{|X|}(G^{0}).$$
Let $A \in \mathsf{K}_{{\rm \tiny fin}} (\mathscr{G})$.
If $A$ is the empty set then we map it to the $|X| \times |X|$ zero matrix.
If $A$ is non-empty then define $\theta (A)$ to be the $|X| \times |X|$ matrix with entries from $G^{0}$ where
$(x,g,y) \in A$ if and only if the $(x,y)$-element of $\theta (A)$ is $g$;
all other entries are zero.
Clearly, $\theta (A)$ has only a finite number of non-zero entries.
Because $A$ is a local bisection, each row of $\theta (A)$ contains at most one non-zero element;
likewise, each column of $\theta (A)$ contains at most one non-zero element.
This defines $\theta$ which is clearly an injection.
We prove that it is a bijection.
Let $A' \in M_{|X|}(G^{0})$.
If it is the zero matrix then we define $A$ to be the empty set.
Then we define $A$ as follows: $(x,g,y) \in A$ precisely when the $(x,y)$-element of $\theta (A)$ is $g$.
Because $A'$ is a rook matrix, it follows that $A$ is a finite local bisection.
We have therefore proved that $\theta$ is a bijection.
It remains to show that it is a homomorphism.
Let $A,B \in \mathsf{K}_{{\rm \tiny fin}} (\mathscr{G})$.
Then it is routine to check that $\theta (AB) = \theta (A)\theta (B)$.
\end{proof}

%%%%%%%%%%%%%%%%%%%%%%%%%%%%%%%%%%%%%%%%%%%%%%%%%%
\begin{center}
{\bf Notation}
\end{center}

\begin{itemize}

\item If $S$ is an inverse semigroup, then $\mathsf{E}(S)$ denotes the semilattice of idempotents of $S$.
If $X \subseteq S$ then $\mathsf{E}(X) = \mathsf{E}(S) \cap X$.

\item If $S$ is an inverse semigroup, then $\mathsf{A}(S)$ denotes the set of atoms of $S$.

\item If $X$ is a finite set, then $\mathscr{I}(X)$ is the symmetric inverse monoid on $X$.
If $X$ has $n$ elements then the symmetric inverse monoid on $X$ is also denoted by $\mathscr{I}_{n}$.
These semigroups are finite Boolean inverse monoids.

\item If $X$ is an infinite, set then $\mathscr{I}_{\rm \tiny fin}(X)$ denotes the Boolean inverse semigroup
of all partial bijections of $X$ with finite domains.

\item If $S$ is an inverse semigroup then $\mathcal{G}(S)$ is called the {\em restricted groupoid} of $S$.
It is equal to the set $S$ with the restricted product if $S$ does not have a zero and the set $S^{\ast}$ with respect
to the restricted product if $S$ does have a zero.

\item $\mathsf{K}(G)$ is the Boolean inverse monoid of all local bisections of the finite discrete groupoid $G$.

\item $\mathsf{K}_{\rm \tiny fin}(G)$ is the Boolean inverse semigroup of all finite local bisections of the discrete groupoid $G$.

\item $\mathsf{B}(S)$ is the Booleanization of the inverse semigroup $S$.

\item $\mathsf{G}_{u}(S)$ is Paterson's universal groupoid of the inverse semigroup $S$.

\item $\mathsf{G}(S)$ is the \'etale groupoid of all prime filters of the Boolean inverse semigroup $S$.

\item $\mathsf{KB}(G)$ is the Boolean inverse semigroup of compact-open local bisections of the Boolean groupoid $G$.

\end{itemize}

%%%%%%%%%%%%%%%%%%%%%%%%%%%%%%%%%%%%%%%%%%%%%%%%%%%%%%%%%%%%%%%%%%%%%%%%%%%%%%%%%%%%%
\section{Finite Boolean inverse monoids}

\subsection{The structure of finite Boolean inverse monoids}
Finite Boolean inverse {\em semigroups} are, of course, finite Boolean inverse {\em monoids}.
The following is a corollary of 
Proposition~\ref{prop:boolean-finite-local-bisections} and will motivate the whole of this section;
in fact, we shall prove that every finite Boolean inverse monoid
is of this form.
Observe that the natural partial order in $\mathsf{K}(G)$ is subset inclusion and so the atoms
are the local bisections which are singleton sets.

\begin{proposition}\label{prop:groupoids} 
Let $G$ be a finite (discrete) groupoid.
Then $\mathsf{K}(G)$, the set of all local bisections of $G$ under subset multiplication,
is a finite Boolean inverse monoid, the set of atoms of which forms a groupoid isomorphic to $G$.
\end{proposition} 

Let $S$ be a finite Boolean inverse monoid.
Our goal is to describe the structure of $S$ in terms of its sets of atoms, just as in the case of finite Boolean algebras.
Whereas in finite Boolean algebras the atoms simply form a set, in Boolean inverse monoids the set of atoms forms a groupoid.
Denote by $\mathsf{A}(S)$ the set of atoms of $S$.
The proof of the following is immediate by finiteness.

\begin{lemma}\label{lem:atoms} Let $S$ be a finite Boolean inverse monoid.
Then each non-zero element is above an atom.
\end{lemma}

The set of atoms is not merely just a set but in fact a groupoid.

\begin{lemma}\label{lem:atoms-groupoid} Let $S$ be an inverse semigroup with zero.
Then the set of atoms, if non-empty, forms a groupoid under the restricted product.
\end{lemma}
\begin{proof} Let $a, b \in \mathsf{A}(S)$ such that $\mathbf{d}(a) = \mathbf{r}(b)$.
We prove that $a \cdot b$ is an atom.
Let $x \leq a \cdot b$.
Then $x = a' \cdot b'$ where $a' \leq a$ and $b' \leq b$
by Lemma~\ref{lem:restricted-product}.
Now, $a$ and $b$ are atoms.
Thus either $x = a \cdot b$ or $x = 0$.
It follows that $a \cdot b$ is an atom as well.
\end{proof}

We now connect elements with the atoms beneath them.
Let $a \in S$.
Define $\theta (a) = a^{\downarrow} \cap \mathsf{A}(S)$.

\begin{lemma}\label{lem:two} 
Let $S$ be a Boolean inverse semigroup.
For each $a \in S$, the set $\theta (a)$ is a local bisection of the groupoid $\mathsf{A}(S)$.
\end{lemma} 
\begin{proof} If $a = 0$ then $\theta (a) = \varnothing$.
If $a \neq 0$ then it is above at least one atom by Lemma~\ref{lem:atoms} and so is non-empty.
Let $x,y \in \theta (a)$ such that $\mathbf{d}(x) = \mathbf{d}(y)$.
Since $x,y \leq a$ we deduce that $a = b$ by Lemma~\ref{lem:l-and-r-order}.
Dually, if $\mathbf{r}(x) = \mathbf{r}(y)$ then $x = y$.
\end{proof}

If $S$ is a finite Boolean inverse monoid, then we have defined a function $\theta \colon S \rightarrow \mathsf{K}(\mathsf{A}(S))$.
Observe that $\theta (0) = \varnothing$.
Our first main theorem is the following;
it shows that the finite Boolean inverse monoids described in Proposition~\ref{prop:groupoids} 
are typical.

\begin{theorem}\label{them:main-finite}Let $S$ be a finite Boolean inverse monoid.
Then $S$ is isomorphic to the Boolean inverse monoid $\mathsf{K}(\mathsf{A}(S))$.
\end{theorem}
\begin{proof} By Lemma~\ref{lem:atoms} and Lemma~\ref{lem:atoms-groupoid},
the set $\mathsf{A}(S)$ of atoms of $S$ forms a groupoid.
By Lemma~\ref{lem:two}, the function $\theta$ is well-defined.
We shall prove that $\theta$ is an isomorphism.
First, $\theta$ is a homomorphism.
Let $x$ be an atom such that $x \leq ab$.
Then $x = a' \cdot b'$ where $a' \leq a$ and $b' \leq b$ by Lemma~\ref{lem:restricted-product}.
It follows that $a'$ and $b'$ are atoms by Lemma~\ref{lem:atom-idempotent}.
We have therefore proved that $\theta (ab) \subseteq \theta (a)\theta (b)$.
Conversely, let $x \in \theta (a)$ and $y \in \theta (b)$ such that the restricted product $x \cdot y$ is defined.
Then $x \cdot y = xy \leq ab$ by Lemma~\ref{lem:restricted-product}.
It remains to prove that $\theta$ is a bijection.
We show first that $a = \bigvee \theta (a)$.
Put $b = \bigvee \theta (a)$.
Then, clearly, $b \leq a$.
Suppose that $b \neq a$.
Then $a \setminus b \neq 0$.
Thus $a \setminus b$ is above an atom $x$ by Lemma~\ref{lem:atoms}.
But then $x \leq a$.
It follows that $x \leq b$ which is a contradiction.
Thus $a = \bigvee \theta (a)$.
It follows that $\theta$ is an injection.
Now let $A \in \mathsf{K}(\mathsf{A}(S))$.
Then $A$ is a set of pairwise orthogonal elements.
Put $a = \bigvee A$.
Clearly, $A \subseteq \theta (a)$.
Let $x$ be an atom such that $x \leq a$.
Then $x = \bigvee_{a \in A} (x \wedge a)$ by Lemma~\ref{lem:eggs}.
It follows that $x = a$ for some $a \in A$.
We have therefore proved that $\theta (a) = A$.
\end{proof}

Let $G$ be a finite groupoid.
We now relate the structure of the Boolean inverse monoid
$\mathsf{K}(G)$ to the structure of the finite groupoid $G$.

\begin{lemma}\label{lem:bordeaux1} \mbox{}
\begin{enumerate}

\item A finite Boolean inverse monoid is fundamental if and only its groupoid of atoms is principal.

\item A finite Boolean inverse monoid is $0$-simplifying if and only if its groupoid of atoms is connected.

\end{enumerate}
\end{lemma}
\begin{proof} 
(1) Let $S$ be fundamental. 
Suppose that 
$e \stackrel{a}{\longrightarrow} e$ 
where $e$ is an atomic idempotent.
Then $a$ is an atom by Lemma~\ref{lem:atom-idempotent} 
Let $f$ be any idempotent.
Then $fa \leq a$.
It follows that $fa = 0$ or $fa = a$.
Suppose that $fa = 0$.
Then $fe = 0$ and so $af = 0$.
Thus $fa = af$.
Suppose now that $fa = a$.
Then $fe = e = fe$ and so $af = a$.
It follows again that $fa = af$.
We have therefore proved that $a$ commutes with every idempotent.
and so, by assumption, $a = e$. 
Conversely, suppose that $e \stackrel{a}{\longrightarrow} e$, where $e$ is an atomic idempotent, implies that $a = e$.
Let $a$ commute with all idempotents.
We prove that $S$ is fundamental by showing that $a$ is an idempotent.
We can write $a = \bigvee_{i=1}^{m} a_{i}$ where the $a_{i}$ are atoms
and by Proposition~\ref{prop:definition} we can assume this is an orthogonal join.
We prove that $\mathbf{d}(a_{i}) = \mathbf{r}(a_{i})$ for all $i$ from which the result follows.
Since $\mathbf{r}(a_{j})a = a \mathbf{r}(a_{j})$ 
we have that
$a_{j} = \bigvee_{i=1}^{m}a_{i}\mathbf{r}(a_{j})$.
But $a_{i}\mathbf{r}(a_{j}) \leq a_{j}$.
Thus either $a_{i}\mathbf{r}(a_{j}) = 0$ or $a_{i}\mathbf{r}(a_{j}) = a_{j}$.
But $a_{i}\mathbf{r}(a_{j}) \leq a_{i}$.
It follows that $a_{i}\mathbf{r}(a_{j}) = a_{i}$ and so $a_{i} = a_{j}$.
Thus $a_{j}\mathbf{r}(a_{j}) = a_{j}$.
Hence $\mathbf{d}(a_{j}) \leq \mathbf{r}(a_{j})$ and so $\mathbf{d}(a_{j}) = \mathbf{r}(a_{j})$.
By assumption $a_{j}$ is an idempotent.
It follows that $a$ is an idempotent.

(2) Let $S$ be a finite Boolean inverse monoid.
Suppose first that it is $0$-simplifying and let $e$ and $f$ be any two atoms.
Then, by assumption and Lemma~\ref{lem:toby}, we have that $e \equiv f$.
In particular, $f \preceq e$.
There is therefore a pencil from $f$ to $e$.
But, since $f$ and $e$ are both atoms, it follows that $f \, \mathscr{D} \, e$.
Thus all atomic idempotents are $\mathscr{D}$-related.
Conversely, let $S$ be a finite Boolean inverse monoid in which all its atomic idempotents are $\mathscr{D}$-related.
Let $e$ and $f$ be any two non-zero idempotents of $S$.
Since $S$ is finite, the idempotent $e$ is the join of all the atoms that lie beneath it, as is $f$.
Let the atoms beneath $e$ be $e_{1}, \ldots, e_{n}$.
Then each $e_{i}$ is $\mathscr{D}$-related to an atom beneath $f$.
But this shows that $e \preceq f$.
By symmetry, we deduce that $e \equiv f$.
\end{proof}

The following combines \cite[Theorem~4.18]{Law5} 
and 
\cite{Malandro2010}.

\begin{theorem}\label{them:finite} Let $S$ be a finite Boolean inverse monoid.
Then there are finite groups $G_{1}, \ldots, G_{r}$ and natural numbers $n_{1}, \ldots, n_{r}$ such that
$$S \cong M_{n_{1}}(G_{1}^{0}) \times \ldots \times M_{n_{r}}(G_{r}^{0}).$$
\end{theorem}
\begin{proof} Let $S$ be a finite Boolean inverse monoid with groupoid of atoms $G$.
Then $S \cong \mathsf{K}(G)$ by Theorem~\ref{them:main-finite}.
Let $G_{1}, \ldots, G_{r}$ be the finite set of connected components of $G$.
Then
$S \cong \mathsf{K}(G_{1}) \times \ldots \times \mathsf{K}(G_{r})$ since the intersection of 
a local bisection with each connected component $G_{i}$ is a local bisection of $G_{i}$.  
Each groupoid $G_{i}$ is connnected.
The result now follows by Proposition~\ref{prop:local-bisections-rook}.
\end{proof}

If we combine the above theorem with Lemma~\ref{lem:bordeaux1} and Example~\ref{ex:solomon},
we obtain the following.
Parts (2) and (3) were first proved in \cite{Law5}.

\begin{corollary}\label{cor:finite-stuff}\mbox{}
\begin{enumerate}

\item The finite $0$-simplifying Boolean inverse monoids are isomorphic to $M_{n}(G^{0})$ where $G$ is a finite group.

\item The finite fundamental Boolean inverse monoids are isomorphic to the finite direct products $\mathscr{I}_{n_{1}} \times \ldots \times \mathscr{I}_{n_{r}}$.

\item The finite simple Boolean inverse monoids are isomorphic to $\mathscr{I}_{n}$.

\end{enumerate}
\end{corollary}

%%%%%%%%%%%%%%%%%%%%%%%%%%%%%%%%%%%%%%%%%%%%%%%%%%%%%%%%%%%%%%%%
\subsection{Booleanizations of finite inverse semigroups}

We shall now relate arbitrary finite inverse semigroups
to finite Boolean inverse monoids.
Let $S$ be an inverse semigroup with zero.
Then there is a universal homomorphism $\beta \colon S \rightarrow \mathsf{B}(S)$ to the category
of Boolean inverse semigroups by \cite{Lawson2020}.
The semigroup $\mathsf{B}(S)$ is called the {\em Booleanization} of $S$. 
If $S$ does not have a zero then adjoin one to get $S^{0}$ and then the Booleanization of $S$ is 
the same as the Booleanization of $S^{0}$.

We now compute the Booleanization of a finite inverse semigroup.
We shall need the following notation.
Let $S$ be any inverse semigroup
and
let $a \in S$.
Define $\hat{a} = a^{\downarrow} \setminus \{a \}$; this is just the set of all elements below $a$ but excluding $a$

\begin{lemma}\label{lem:homo} Let $S$ be a finite inverse semigroup and let $\alpha \colon S \rightarrow T$
be a homomorphism to a Boolean inverse semigroup $T$.
Let $a,b \in S$ be any elements such that 
$\mathbf{d}(a) \neq \mathbf{d}(b)$ and
$\mathbf{r}(a) \neq \mathbf{r}(b)$;
this means precisely that neither $a^{-1}b$ nor $ab^{-1}$ are restricted products.
Let 
$\hat{a} = \{a_{1}, \ldots, a_{m} \}$
and 
$\hat{b} = \{b_{1}, \ldots, b_{n} \}$.
Then the elements
$x = \alpha (a) \setminus (\alpha (a_{1}) \vee \ldots \vee \alpha (a_{m}))$
and
$y =\alpha (b) \setminus (\alpha (b_{1}) \vee \ldots \vee \alpha (b_{n}))$
are orthogonal.
\end{lemma}
\begin{proof} We shall prove that $x^{-1}y = 0$;
the proof that $xy^{-1} = 0$ follows by symmetry.
We have that
$$x^{-1}y = \alpha (a^{-1}b) \setminus (\alpha(a^{-1}b_{1}) \vee \ldots \vee \alpha (a^{-1}b_{n}) \vee \alpha (a_{1}^{-1}b) \vee \ldots \vee \alpha (a_{m}^{-1}b)).$$
We know that $a^{-1}b$ cannot be a restricted product.
It follows by Lemma~\ref{lem:restricted-product} that
either $a^{-1}b = a^{-1}b_{j}$, for some $j$, or $a^{-1}b = a_{i}^{-1}b$, for some $i$.
In both cases, it is immediate that $x^{-1}y = 0$.
\end{proof}

\begin{lemma}\label{lem:important} Let $S$ be a Boolean inverse semigroup.
Let $a \in S$ and $a_{1}, \ldots, a_{m} < a$.
Put $X = \{a, a_{1},\ldots, a_{m}\}$.
We regard $X$ as a poset with respect to the natural partial order.
For each $x \in X$ put $\tilde{x} = (x^{\downarrow} \cap X) \setminus \{x\}$.
If $x$ is minimal in $X$ then $\tilde{x}$ is empty.
Thus $\tilde{x}$ is the set of all elements in $X$ below $x$ excluding $x$ itself.
Then $a = \bigvee_{x \in X} (x \setminus \bigvee \tilde{x})$.
\end{lemma}
\begin{proof} Define the {\em height} of an element of $X$ as follows:
the minimal elements of $X$ have height $0$.
When $x$ is minimal $(x \setminus \bigvee \tilde{x}) = x$.
The elements of $X$ that cover elements of height $n$ have height $n+1$.
Consider the elements of $X$ with height $1$.
Then for such an element $x$ we will have elements of the form $x \setminus (x_{1} \vee \ldots \vee x_{r})$
where $x_{1},\ldots, x_{r}$ are minimal.
But 
$$(x \setminus (x_{1} \vee \ldots \vee x_{r})) \vee x_{1} \vee \ldots \vee x_{r} = x$$
by part (3) of Lemma~\ref{lem:chicken}.
We now use induction.
\end{proof}

\begin{theorem}\label{them:booleanization-finite}
Let $S$ be a finite inverse semigroup.
Then the Booleanization of $S$ is $\mathsf{K}(\mathcal{G}(S))$,
where $\mathcal{G}(S)$ is the restricted groupoid of $S$.
\end{theorem}
\begin{proof}  Put $G = \mathcal{G}(S)$ to ease notation.
Define $\beta \colon S \rightarrow \mathsf{K}(G)$ by
$\beta (a) = a^{\downarrow} \cap G$.
Observe that if $a = 0$ then $\beta (0) = \varnothing$.
It is immediate that $\beta$ is an injective homomorphism by Lemma~\ref{lem:restricted-product}. 
We now prove that this map is universal to Boolean inverse semigroups.
Let $\alpha \colon S \rightarrow T$ be any homomorphism to a Boolean inverse semigroup $T$.
We shall define a map $\gamma \colon \mathsf{K}(G) \rightarrow T$.
Define $\gamma (\varnothing) = 0$.
Let $a \in S$ be nonzero and let $\hat{a} = \{a_{1},\ldots , a_{m}\}$.
Observe that $\{a \} \in  \mathsf{K}(G)$.
Define 
$$\gamma (\{a\}) = \alpha(a) \setminus (\alpha (a_{1}) \vee \ldots \vee \alpha (a_{m})),$$
which is clearly well-defined.
We now extend $\gamma$.
Let $a, b \in A$ where $A$ is a local bisection in $G$.
Then $\gamma (\{a\})$ and $\gamma (\{b\})$ are orthogonal by Lemma~\ref{lem:homo}.
We may therefore define $\gamma (A) = \bigvee_{a \in A} \gamma (\{a\})$.

We have to prove that $\gamma \beta = \alpha$, 
that $\gamma$ is a morphism of Boolean inverse semigroups
and that $\gamma$ is unique with these properties.
We first prove that $\gamma (\beta (a)) = \alpha (a)$.
We therefore have to compute $\gamma (a^{\downarrow})$.
By definition $\gamma (a^{\downarrow}) = \bigvee_{x \leq a} \gamma (\{x\})$. 
Let $\hat{a} = \{a_{1},\ldots, a_{m}\}$.
Then $\gamma (a^{\downarrow}) = \gamma (\{a\}) \vee \gamma (\{a_{1}\}) \vee \ldots \vee \gamma (\{a_{m}\})$.
Now, this equals $\alpha (a)$ by Lemma~\ref{lem:important}.
We prove that $\gamma$ is a homomorphism.
This follows from Lemma~\ref{lem:restricted-product}
and Lemma~\ref{lem:properties-of-setminus}.
It clearly maps binary joins to binary joins.

We need to prove $\gamma$ is unique with the given properties.
Let $\gamma' \colon \mathsf{K}(G) \rightarrow T$ be a morphism that satisfies the properties.
Let $a$ be a non-zero element.
By assumption, $\gamma' (a^{\downarrow}) = \alpha (a)$.
Let $a_{1},\ldots, a_{m} < a$ be all the elements of $S$ below $a$.
Then $\{a\} = \beta (a) \setminus \{a_{1},\ldots, a_{m}\}$ 
where both $\beta (a)$ and $\{a_{1},\ldots,a_{m}\}$ are local
bisections.
Then $\gamma' (\{a\}) = \gamma' (\beta (a)) \setminus \gamma' (\{a_{1}, \ldots, a_{m} \}) =  \alpha (a) \setminus \gamma' (\{a_{1}, \ldots, a_{m} \})$
But $\{a_{1},\ldots a_{m}\} = \beta (a_{1}) \cup \ldots \cup \beta (a_{m})$,
and the result now follows.
\end{proof}

\begin{theorem}\label{them:groupoids-Booleanizations} Let $S$ and $T$ be finite inverse semigroups.
\begin{enumerate}
\item If $S$ and $T$ do not have zeros then their Booleanizations are isomorphic if and only if the groupoids $S^{\cdot}$
and $T^{\cdot}$ are isomorphic.
\item If $S$ and $T$ both have zeros then their Booleanizations are isomorphic if and only if 
their restricted groupoids are isomorphic.
\end{enumerate}
\end{theorem}
\begin{proof} We prove (1) since the proof of (2) is similar.
Suppose that $\mathsf{B}(S) \cong \mathsf{B}(T)$.
Then by Theorem~\ref{them:booleanization-finite}, we have that 
$\mathsf{B}(S) \cong \mathsf{K}(\mathcal{G}(S))$ 
and  
 $\mathsf{B}(T) \cong \mathsf{K}(\mathcal{G}(T))$. 
Thus $\mathsf{K}(\mathcal{G}(S)) \cong \mathsf{K}(\mathcal{G(T)})$.
By Proposition~\ref{prop:groupoids}, 
the groupoids $\mathcal{G}(S)$ and $\mathcal{G}(T)$ are isomorphic
and so, in both cases,
the groupoids $S^{\cdot}$ and $T^{\cdot}$ are isomorphic, as claimed.
The proof of the converse is immediate.
\end{proof}

%%%%%%%%%%%%%%%%%%%%%%%%%%%%%%%%%%%%%%%%%%%%%%%%%%%%%%%%%%%%
\subsection{The representation theory of finite inverse monoids}

We shall now consider the representation theory of finite inverse monoids in the light of the above results.
We shall assume that our inverse semigroups do not have a zero.
One of the interesting features of the representation theory of finite inverse semigroups $S$
is that it is only the groupoid $S^{\cdot}$ that is important;
see \cite[Theorem 9.3]{Steinberg}.
We can now explain why this must be true.
Let $\theta \colon S \rightarrow M_{n}(k)$ be a representation of the inverse semigroup $S$
where $k$ is a field.
We know that $M_{n}(k)$ is a ring.
Thus by \cite[Proposition 3.2]{Lawson2020}, the image of $\theta$ is contained in a Boolean inverse submonoid of the multiplicative monoid
of $M_{n}(k)$.
It follows that $\theta$ is essentially a homomorphism from an inverse semigroup to a Boolean inverse monoid.
Thus there is a homomorphism $\theta^{\ast} \colon \mathsf{B}(S) \rightarrow M_{n}(k)$
which extends $\theta$;
in addition, $\theta^{\ast}$ is an {\em additive representation} meaning that if $a \perp b$ then $\theta^{\ast}(a) + \theta^{\ast}(b) = \theta^{\ast}(a \oplus b)$.
But by Theorem~\ref{them:booleanization-finite}, we have seen that $\mathsf{B}(S) \cong \mathsf{K}(S^{\cdot})$.
Thus the representation theory of finite inverse semigroups is determined by their restricted groupoids.
%By Theorem~\ref{them:finite}, we are therefore interested in additive representations of
%$M_{n_{1}}(G_{1}^{0}) \times \ldots \times M_{n_{r}}(G_{r}^{0})$
%in $M_{n}(k)$; of course, the left-hand side are rook matrices whereas the right-hand side are matrices over rings.

%Observe that the Boolean inverse monoid of rook matrices $M_{n}(G^{0})$, where $G^{0}$ is the finite group $G$ with a zero adjoined,
%can be embedded into the multiplicative monoid of the ring $M_{n}(kG)$, where $kG$ is the group-algebra of $G$ over the field $k$.
%It can be proved that  $kM_{n}(G^{0}) \cong M_{n}(kG)$.
%To see why, observe that the $kM_{n}(G^{0})$ contains a complete set of matrix units:
%these being the $n \times n$ matrices which are zero everywhere but have the identity of $G^{0}$ in exactly
%one position.
%The centralizer in $kM_{n}(G^{0})$ of all the matrix units are the elements 
%of the form $aI$ where $I$ is the $n \times n$ identity  and $a \in kG$.

 %Now, let $S$ be the finite Boolean inverse monoid where
%$S \cong M_{n_{1}}(G_{1}^{0}) \times \ldots \times M_{n_{r}}(G_{r}^{0})$
%such that $G_{1},\ldots, G_{r}$ are finite groups.
%Then 
%$kS \cong M_{n_{1}}(kG_{1}) \times \ldots \times M_{n_{r}}(kG_{r})$.

%%%%%%%%%%%%%%%%%%%%%%%%%%%%%%%%%%%%%%%%%%%%%%%%%%%%%%%%%%%%%%%%%%%%%%%%%%%
\section{Semisimple  Boolean inverse semigroups}

We say that an inverse semigroup $S$ is {\em semisimple} if for each $a \in S$ the set $a^{\downarrow}$ is finite.
Since $a^{\downarrow}$ and $\mathbf{d}(a)^{\downarrow}$ (respectively, $\mathbf{r}(a)^{\downarrow}$) are order-isomorphic,
it follows that the set $a^{\downarrow}$ is finite precisely when the set $\mathbf{d}(a)^{\downarrow}$ is finite (respectively, the set 
$\mathbf{r}(a)^{\downarrow}$ is finite).
We therefore have the following.

\begin{lemma}\label{lem:when-semisimple} 
An inverse semigroup $S$ is semisimple if and only if the sets $e^{\downarrow}$ are finite for
all idempotents $e$ of $S$.
\end{lemma}

\begin{remark}{\em Lemma~\ref{lem:when-semisimple} reassures us that to determine whether a Boolean inverse semigroup
is semisimple or not we need only look at the idempotents.}
\end{remark}

\begin{proposition}\label{prop:atoms-semisimple} In a Boolean inverse semigroup $S$, the following assertions are equivalent:
\begin{enumerate}
\item $S$ is semisimple.
\item Each non-zero element is above an atom and every element is a finite join of atoms.
\end{enumerate}
\end{proposition}
\begin{proof} (1)$\Rightarrow$(2).
Let $a$ be any non-zero element.
Then the set $a^{\downarrow}$ is finite.
Either $a$ is an atom or is strictly above a non-zero element $b < a$.
Now repeat the above argument with $b$.
Since the poset is finite we shall eventually end at an atom that is below $a$.
The set $a^{\downarrow}$ is finite and so there are atoms below $a$.
Let $a_{1},\ldots, a_{m}$ be all the atoms below $a$.
These elements are pairwise compatible because they are all below the same element.
We prove that $a = \bigvee_{i=1}^{m} a_{i}$.
Put $b = \bigvee_{i=1}^{m} a_{i}$.
Clearly, $b \leq a$.
Suppose that $b < a$.
Then $a \setminus b$ is nonzero.
Let $a' \leq a \setminus b$ be an atom.
Then $a' = a_{i}$ for some $i$
implying that $a' \leq b$.
But $b \perp (a \setminus b)$ and so we get a contradiction.
It follows that $b = a$, and we have proved the result.

(1)$\Rightarrow$(2). Let $a$ be a non-zero element.
Suppose that $a = \bigvee_{i=1}^{m} a_{i}$ is a finite join of atoms.
We prove that the set $a^{\downarrow}$ is finite.
Let $b \leq a$.
Then $b = a\mathbf{d}(b) = \bigvee_{i=1}^{m} a_{i}\mathbf{d}(b)$.
But $a_{i}\mathbf{d}(b) = 0$ or $a_{i}\mathbf{d}(b) = a_{i}$ since $a_{i}$ is an atom.
It we omit all products where $a_{i}\mathbf{d}(b) = 0$
then we have proved that $b$ is a join of some of the atoms in the set $\{a_{1},\ldots, a_{m}\}$.
It follows that there are only a finite number of elements below $a$ and so $S$ is semisimple.
\end{proof}

The following is a special case of \cite[Proposition 3.7.8]{Wehrung}.

\begin{lemma}\label{lem:meets-semisimple} 
Semisimple Boolean inverse semigroups are meet semigroups.
\end{lemma}
\begin{proof} Let $a$ and $b$ be any two elements of the semisimple Boolean inverse semigroup $S$.
Then $a^{\downarrow} \cap b^{\downarrow}$ is a finite set of compatible elements and so its join exists.
This join is precisely $a \wedge b$.
\end{proof}

\begin{remark}{\em In the light of Lemma~\ref{lem:meets-semisimple},
observe that the binary meets in semisimple Boolean inverse semigroups are compatible meets and so automatically preserved by homomorphisms.
We deduce from Proposition~\ref{prop:anja} that all morphisms with domains semisimple Boolean inverse semigroups
are ideal-induced.
This observation strengthens part of \cite[Proposition 3.7.6]{Wehrung}.}
\end{remark}

In this section, we shall describe the structure of semisimple Boolean inverse semigroups.
Of course, all finite Boolean inverse monoids are semisimple but we shall use more sophisticated methods in this section than in the previous one.
The terminology `semisimple' is due to \cite{Wehrung}.
We now recall the elements of the theory of non-commutative Stone duality we shall need.

A {\em Boolean space} is a $0$-dimensional locally compact Hausdorff space.
For \'etale groupoids see \cite{Resende}.
An \'etale groupoid is said to be {\em Boolean} if its space of identities is a Boolean space.
Let $G$ be a Boolean groupoid.
Denote by $\mathsf{KB}(G)$ the set of all compact-open local bisections of $G$.
Then $\mathsf{KB}(G)$ is a Boolean inverse semigroup.
In this section, we shall use filters of various complexions.
Let $S$ be an inverse semigroup.
A subset $F \subseteq S$ is called a {\em filter} if $a,b \in F$ implies that there exists $c \in F$ such that $c \leq a,b$
and $F^{\uparrow} = F$.
A filter $F$ is {\em proper} if $0 \notin F$.
Let $F$ be a filter in $S$.
Define $\mathbf{d}(F) = (A^{-1}A)^{\uparrow}$ and $\mathbf{r}(A) = (AA^{-1})^{\uparrow}$;
both of these sets are filters.
A maximal proper filter is called an {\em ultrafilter}.
Let $S$ be a distributive inverse semigroup.
A proper filter $F$ is said to be {\em prime} if $a \vee b \in F$ implies that $a \in F$ or $b \in F$.
The set of prime filters containing the element $a$ is denoted by $V_{a}$. 
The set of prime filters $\mathsf{G}(S)$ is a Boolean groupoid when $S$ is a Boolean inverse semigroup.
Non-commutative Stone duality states that when $S$ is a Boolean inverse semigroup then $S \cong \mathsf{KB}(\mathsf{G}(S))$
and when $G$ is a Boolean group then $G \cong \mathsf{G}(\mathsf{KB}(G))$.
See \cite{LL}.
The following is \cite[Lemma 3.20]{LL}.

\begin{proposition}\label{prop:semigroup-filters} Let $S$ be a distributive inverse semigroup.
The semigroup $S$ is Boolean if and only if every prime filter is an ultrafilter.
\end{proposition}

%%%%%%%%%%%%%%%%%%%%%%%%%%%%%%%%%%%%%%%%%%%%%%%%%%%%%%%%%%%%%%%%%%%%%%%%%%%%
\subsection{Structure theorems}

Our first result is a topological criterion for a Boolean inverse semigroup to be semisimple.

\begin{theorem}\label{them:discrete-topology}
Let $S$ be a Boolean inverse semigroup.
\begin{enumerate}
\item When $S$ is a semisimple,
the groupoid $\mathsf{G}(S)$ is isomorphic with the groupoid of atoms of $S$
under the restricted product.
\item $S$ is semisimple if and only if the topology on $\mathsf{G}(S)$ is discrete.
\end{enumerate}
\end{theorem}
\begin{proof} (1) Let $S$ be a semisimple Boolean inverse semigroup.
We use Proposition~\ref{prop:semigroup-filters} which tells us that in Boolean inverse semigroups
prime filters are the same as ultrafilters. 
We prove first  that the ultrafilters of $S$ are of the form $a^{\uparrow}$ where $a$ is an atom.
Let $a$ be an atom and suppose that $a^{\uparrow} \subseteq A$ where $A$ is any proper filter.
Let $b \in A$.
But $A$ is a filter and so there exists $c \in A$ such that $c \leq a,b$.
But $a$ is an atom.
Thus $c = a \leq b$.
It follows that $a^{\uparrow} = A$.
We have therefore proved that $a^{\uparrow}$ is an ultrafilter.
Since any filter containing $a$ must contain $a^{\uparrow}$ as a subset, it follows
that $a^{\uparrow}$ is the only ultrafilter containing the atom $a$.
Now, let $A$ be an ultrafilter.
Let $b \in A$ be arbitrary.
The element $b$ is a join of atoms by Proposition~\ref{prop:atoms-semisimple}
and so, since $A$ is also a prime filter, $A$ contains at least one atom, say $a$.
Observe that $A$ contains exactly one atom since if $a'$ is any atom in $A$ then there is $c \leq a,a'$ where $c \in A$.
It follows that $c = a = a'$.
Clearly $a^{\uparrow} \subseteq A$, but we proved above that $a^{\uparrow}$ is an ultrafilter so $a^{\uparrow} = A$.
We have proved that there is a bijective correspondence between atoms of $S$ and ultrafilters of $S$.
Denote the set of atoms of $S$ by $\mathsf{A}(S)$.
This is a groupoid under the restricted product.
It is routine to check that the restricted product $a \cdot b$ exists in $\mathsf{A}(S)$
if and only if the product $a^{\uparrow} \cdot b^{\uparrow}$ exists in  $\mathsf{G}(S)$;
in addition, $a^{\uparrow} \cdot b^{\uparrow} = (a \cdot b)^{\uparrow}$.
Thus $\mathsf{A}(S)$ and $\mathsf{G}(S)$ are isomorphic as groupoids.

(2) Let $S$ be a semisimple Boolean inverse semigroup.
By (1) above, observe that if $a$ is an atom then $V_{a} = \{a^{\uparrow}\}$.
Thus the topology on $\mathsf{G}(S)$ is discrete.
Conversely, let $S$ be a Boolean inverse semigroup
where the groupoid $\mathsf{G}(S)$ is equipped with the discrete topology. 
We prove that $S$ is semisimple.
Let $A$ be any compact-open local bisection.
In a discrete space, the only compact sets are finite.
Thus $A$ is finite.
By non-commutative Stone duality, 
$\mathsf{KB}(\mathsf{G}(S)) \cong S$.
In $\mathsf{KB}(\mathsf{G}(S))$ the natural partial order is subset inclusion.
Thus there can only be a finite number of elements below $A$.
It follows that $S$ is semisimple.
\end{proof}

By Theorem~\ref{them:discrete-topology} and non-commutative Stone duality,
we now have the following description of semisimple Boolean inverse semigroups in terms of discrete groupoids.

\begin{corollary}\label{cor:semisimple} \mbox{}
The semisimple Boolean inverse semigroups are isomorphic to those of the form
$\mathsf{K}_{\rm \tiny fin}(G)$ 
where $G$ is a discrete groupoid.
\end{corollary}

Let $S$ be a semsimple Boolean inverse semigroup.
Then the groupoid of atoms of $S$, denoted by $\mathsf{A}(S)$, is isomorphic with the usual 
groupoid $\mathsf{G}(S)$ of prime filters of $S$.

The following result will enable us to describe the $0$-simplifying and fundamental Boolean inverse semigroups.
The proof is similar to the proof of Lemma~\ref{lem:bordeaux1}.

\begin{lemma}\label{lem:semisimple-bordeaux} \mbox{}
\begin{enumerate} 
\item A semisimple inverse semigroup is fundamental if and only if its groupoid of atoms is principal.
\item A semisimple Boolean inverse semigroup is $0$-simplifying if and only if its groupoid of atoms is connected.
\end{enumerate}
\end{lemma}

The proof of the following is immediate by Proposition~\ref{prop:local-bisections-rook} 
and Lemma~\ref{lem:semisimple-bordeaux}.

\begin{proposition}\label{prop:semisimple-0-simplifying} 
A Boolean inverse semigroup is $0$-simplifying and semisimple if and only if it is isomorphic to
a Boolean inverse semigroup of the form $M_{|X|}(G^{0})$ for some non-empty set $X$ and group $G$.
\end{proposition}

The following is an immediate corollary.

\begin{corollary}\label{cor:watch} 
The simple, semisimple Boolean inverse semigroups are isomorphic to the semigroups $\mathscr{I}_{\rm {\tiny fin}}(X)$
where $X$ is a non-empty set.
\end{corollary}

%%%%%%%%%%%%%%%%%%%%%%%%%%%%%%%%%%%%%%%%%%%%%%%%%%%%%%%%%%%%%%%%%%%%%
We now turn to the structure of arbitrary semisimple Boolean inverse semigroups.
%%%%%%%%%%%%%%%%%%%%%%%%%%%%%%%%%%%%%%%%%%%%%%%%%%%%%%%%%%%%%%%%%%%%

We denote the {\em (unrestricted) direct product} of the Boolean inverse semigroups $S_{i}$, where $I \in I$, by
$$\prod_{i \in I} S_{i}.$$
This is also a Boolean inverse semigroup.

\begin{example}\label{ex:choco} {\em The product of any finite number of semisimple semigroups is clearly semisimple, but this is not true of arbitrary direct products.
To see why, let $S = \mathscr{I}_{2} \times \mathscr{I}_{3} \times \mathscr{I}_{4} \times \ldots$.
This is a Boolean inverse monoid but clearly the number of elements below the identity element is infinite and so $S$ is not semisimple.}
\end{example}

Define the {\em restricted direct product} of  the Boolean inverse semigroups $S_{i}$, where $I \in I$, 
to be that  subsemigroup of the direct product consisting of those elements
all of whose components are zero except in a finite number of positions;
that is, those elements with {\em finite support}.
This is denoted by
$$\bigoplus_{i \in I} S_{i}.$$
The proof of the following is immediate. The second part should be read in tandem with Example~\ref{ex:choco}.

\begin{lemma}\label{lem:restricted} \mbox{}
\begin{enumerate}
\item The restricted direct product of Boolean inverse semigroups is a Boolean inverse semigroup.
\item The restricted direct product of semisimple semigroups is semisimple.
\end{enumerate}
\end{lemma}

\begin{theorem}\label{them:struct-semisimple}
Every semisimple Boolean inverse semigroup is isomorphic to a restricted direct product
of semisimple Boolean inverse semigroups of the form $M_{|X|}(G^{0})$ where $G$ is a group
and $X$ is a non-empty set.
\end{theorem}
\begin{proof} Let $S$ be a semisimple Boolean inverse semigroup.
By Corollary~\ref{cor:semisimple},
we have that $S \cong \mathsf{K}_{\rm {\tiny fin}}(G)$ where $G$ is a discrete groupoid.
Write $G = \bigsqcup_{i \in I} G_{i}$ where the $G_{i}$ are the connected components of $G$.
Then $S \cong \mathsf{K}_{\rm {\tiny fin}}(\bigsqcup_{i \in I} G_{i})$.
But
$\mathsf{K}_{\rm {\tiny fin}}(\bigsqcup_{i \in I} G_{i}) \cong \bigoplus_{i \in I} \mathsf{K}_{\rm {\tiny fin}}(G_{i})$. 
The result now follows by Proposition~\ref{prop:local-bisections-rook}.
\end{proof}

The following theorem tells us that semisimple Boolean inverse semigroups are basic building blocks
of all Boolean inverse semigroups.

\begin{theorem}[Dichotomy theorem]\label{them:dichotomy} Let $S$ be a $0$-simplifying Boolean inverse semigroup.
Then exactly one of the following is true:
\begin{enumerate}
\item $S$ is atomless.
\item $S$ is semisimple.
\end{enumerate}
\end{theorem}
\begin{proof} Suppose that $S$ contains at least one atomic idempotent $e$.
We prove that every non-zero idempotent is a finite join of atoms from which it will follow that $S$ is semisimple.
Let $f$ be any non-zero idempotent.
Since $S$ is $0$-simplifying, we know that $f \equiv e$.
In particular, $f \preceq e$.
Thus there is a pencil $X = \{x_{1}, \ldots, x_{n} \}$ from $f$ to $e$.
By definition, $f = \bigvee_{i=1}^{m} \mathbf{d}(x_{i})$ and $\mathbf{r}(x_{i}) \leq e$.
But $e$ is an atom and so $\mathbf{r}(x_{i}) = e$.
It follows that $x_{i}$ is an atom and so $\mathbf{d}(x_{i})$ is an atom.
We have therefore proved that $f$ is a join of (some of the) atoms;
we write $f = \bigvee_{i=1}^{m} e_{i}$ where the $e_{i}$ are atoms.
Let $e \leq f$ be any atom.
Then $e = \bigvee_{i=1}^{m} (e_{i}e)$.
Omit all the $i$ such that $e_{i}e = 0$.
Then $e$ is a join of some of the atoms $e_{1},\ldots, e_{m}$ and so must equal one of these atoms.
It follows that there are a finite number of atoms below $f$.
But each element of $f^{\downarrow}$ is a join of a finite number of these atoms.
It follows that the set $f^{\downarrow}$ is finite.
\end{proof}

%%%%%%%%%%%%%%%%%%%%%%%%%%%%%%%%%%%%%%%%%%%%%%%%%%%%%%%%%%%
\subsection{Booleanizations}

In this section, we shall use the {\em universal groupoid} $\mathsf{G}_{u}(S)$ associated with the inverse semigroup $S$.
This was introduced by \cite{Paterson} but we shall adopt the approach to this groupoid introduced by \cite{Lenz} as mediated through \cite{LMS}.
We refer to \cite[Section 5.1]{LL} for the details.
Denote by $\mathsf{G}_{u}(S)$ the groupoid of all (proper) filters of $S$.
The basis of the topology on  $\mathsf{G}_{u}(S)$ is given by sets of the form $U_{a;a_{1},\ldots, a_{m}}$ where $a_{1},\ldots, a_{m} \leq a$ in $S$
which consists of all proper filters that contain $a$ and omit $a_{1},\ldots,a_{m}$.

\begin{theorem}\label{them:universal-groupoid} Let $S$ be an inverse semigroup with zero.
Then $S$ is semisimple if and only if its universal groupoid is discrete.
\end{theorem}
\begin{proof} Assume first that $S$ is a semisimple inverse semigroup.
If $0 \in A$ then $A = 0^{\uparrow}$.
We may therefore assume in what follows that $A$ is a proper filter.
Let $a \in A$.
The set $a^{\downarrow} \cap A$ is non-empty and finite.
This set must have a smallest element $x_{a}$ since $A$ is a filter.
Clearly, $x_{a}^{\uparrow} \subseteq A$.
Let $b \in A$.
Then $a \wedge b \in A$.
But $a \wedge b \in a^{\downarrow} \cap A$.
Thus $x_{a} \leq a \wedge b \leq b$.
We have therefore shown that $A = x_{a}^{\uparrow}$.
Thus every filter is principal.
 Suppose that $a$ covers $a_{1},\ldots, a_{m}$.
 Then the only element of $U_{a;a_{1},\ldots,a_{m}}$ is $a^{\uparrow}$.
 Thus the topology is discrete.
 We now prove the converse.
 Suppose that $\mathsf{G}_{u}(S)$ is a discrete groupoid.
We shall prove that $S$ is semisimple.
Let $A$ be a proper filter.
Then there exists $a_{1},\ldots, a_{m} \leq a$ in $S$ such that
$\{A \} = U_{a;a_{1},\ldots,a_{m}}$.
But this set also contains $a^{\downarrow}$.
Thus $A = a^{\downarrow}$.
We have proved that every filter is principal.
Let $a \in S$.
We need to prove that the set $a^{\downarrow}$ is finite.
Observe that $U_{a}$ is a compact set.
But the topology on $ \mathsf{G}_{u}(S)$ is discrete.
It follows that $U_{a}$ is finite.
But for each $b \leq a$ we have that $b^{\uparrow} \in U_{a}$.
Thus $a^{\downarrow}$ is finite.
\end{proof}

The proof of the following is now straightforward;
recall that the universal groupoid consists of proper filters.

\begin{corollary}\label{cor:carre}
If $S$ is semisimple then the universal groupoid of $S$ is isomorphic to the discrete restricted groupoid $\mathcal{G}(S)$.  
\end{corollary}

By Theorem~\ref{them:universal-groupoid}, Corollary~\ref{cor:carre} and \cite{Lawson2020}, the Booleanization
of the semisimple inverse semigroup $S$ is the Boolean inverse semigroup $\mathsf{K}_{\rm \tiny fin}(\mathscr{G}(S))$.
The map $\beta \colon S \rightarrow \mathsf{K}_{\rm \tiny fin}(\mathscr{G}(S))$ is given by
$\beta (a) = a^{\downarrow} \setminus \{0\}$;
we could of course use Lemma~\ref{lem:homo} and Lemma~\ref{lem:important} but we appeal to general theory.

%%%%%%%%%%%%%%%%%%%%%%%%%%%%%%%%%%%%%%%%%%%%%%%%%%%%%%%%%%%%%%
\section{Type monoids}

We refer the reader to \cite[Chapter 2]{GW} and \cite[Chapter 3]{Wehrung}
for background on commutative refinement monoids.
Let $(S,+,0)$ be a commutative monoid.
It is called a {\em refinement monoid}
if it satisfies the following property:
if $a_{1} + a_{2} = b_{1} + b_{2}$ then
there exist elements $c_{11},c_{12},c_{21},c_{22}$ such that
$a_{1} = c_{11} + c_{12}$ and $a_{2} = c_{21} + c_{22}$,
and
$b_{1} = c_{11} + c_{21}$ and $b_{2} = c_{12} + c_{22}$.
It is said to be {\em conical}
if $a + b = 0$ then $a = 0$ and $b = 0$.
Define $a \leq b$ if and only if $a + c = b$ for some $c$; this is a preorder called the {\em algebraic preordering}.
If $A$ is a preordered set with preorder $\leq$ then a subset $X$ is a {\em lower subset} of $A$ if $x \in X$
and $a \leq x$ implies that $a \in X$.
A submonoid of $S$ which is also a lower subset is called a {\em $o$-ideal}.
A commutative monoid is said to be {\em simple} if it has no non-trivial $o$-ideals.

Let $S$ be a Boolean inverse semigroup
and
let $M$ be a commutative monoid (whose binary operation we write as addition).
A function $\beta \colon \mathsf{E}(S) \rightarrow M$ is called a {\em monoid valuation} if the following conditions hold:
\begin{description}
\item[{\rm (V1)}]  $\beta (0) = 0$.
\item[{\rm (V2)}] $\beta (e \oplus f) = \beta (e) + \beta (f)$ whenever $e \perp f$.
\item[{\rm (V3)}] If $e \mathscr{D} f$ then $\beta (e) = \beta (f)$. 
\end{description}

A Boolean inverse semigroup $S$ is said to be {\em orthogonally separating} if for any idempotents $e,f \in S$
there exist idempotents $e' \, \mathscr{D} \, e$ and $f' \, \mathscr{D} \, f$ such that $e' \perp f'$.
The proof of Lemma~\ref{lem:type-monoid-basics} can be found in \cite[Proposition 2.12]{KLLR2016}.

\begin{lemma}\label{lem:type-monoid-basics} Let $S$ be an orthogonally separating Boolean inverse semigroup.
Denote the elements of $\mathsf{M}(S) = \frac{\mathsf{E}(S)}{\mathscr{D}}$ by $[e]$.
Define $[e] + [f] = [e' \vee f']$ where $e' \mathscr{D} e$, $f' \mathscr{D} f$ and $e' \perp f'$.
Then $\left(\mathsf{M}(S), + , [0] \right)$ is a commutative, conical refinement monoid.
The function $\beta \colon \mathsf{E}(S) \rightarrow \mathsf{T}(S)$ defined by $\beta (e) = [e]$
is a monoid valuation.
\end{lemma}

We now define the type monoid of an arbitrary Boolean inverse semigroup.
We use the following proved as \cite[Lemma 3.9]{KLLR2016}.

\begin{lemma}\label{lem:butterfly} Let $S$ be a Boolean inverse semigroup.
Then $M_{\omega}(S)$ is orthogonally separating.
\end{lemma}

Define $\Delta_{\omega}(a)$ to be the $\omega \times \omega$ matrix the only non-zero entry of which
is an $a$ in the first row and first column.

The type monoid will be a {\em universal} monoid valuation
and is a conical, refinement monoid.
We describe how it is constructed following \cite{KLLR2016}.
Let $S$ be a Boolean inverse semigroup.
Then $M_{\omega}(S)$ is orthogonally separated by Lemma~\ref{lem:butterfly}.
Define $\mathsf{T}(S)$ to be the set $\frac{\mathsf{E}(M_{\omega}(S))}{\mathscr{D}}$
equipped with addition defined by $[e] + [f] = [e' \vee f']$ where $e' \mathscr{D} e$ and $f' \mathscr{D} f$ and $e' \perp f'$.
There is a monoid valuation $\tau \colon \mathsf{E}(S) \rightarrow \mathsf{T}(S)$
called the {\em type function} defined by $\tau (e) = [\Delta_{\omega}(e)]$.
It is the universal monoid valuation.
We shall prove in a later paper that our definition of the type monoid and the one given in \cite{Wehrung}
are really the same.

\begin{remark}{\em 
The type monoid of a Boolean inverse semigroup was introduced in \cite{KLLR2016, Wallis} and its theory further developed in \cite{Wehrung}.
It can be regarded as being the carrier of a notion of `dimension' for idempotents in the sense of von Neumann \cite{VN}
or as a carrier of a generalized notion of cardinality as in Tarski \cite{Tarski}.
}
\end{remark}

The type monoid is an {\em invariant} of a Boolean inverse semigroup but, of course, not necessarily a complete invariant.
In fact, the type monoid is really an invariant of {\em fundamental} Boolean inverse semigroups.
For a proof of the following, see \cite[Theorem 4.4.19]{Wehrung}. 

\begin{proposition}\label{prop:type-fundamental} 
The Boolean inverse semigroups $S$ and $S/\mu$ have isomorphic partial type monoids and isomorphic type monoids.
\end{proposition}

\begin{proposition}\label{prop:type-monoids} 
Let $S$ be a Boolean inverse semigroup.
\begin{enumerate}
\item There is an order-iso\-morphism between the set of additive ideals of $S$ and the additive ideals of  $M_{\omega}(S)$.
In particular, each additive ideal of $M_{\omega}(S)$ is of the form $M_{\omega}(I)$ where $I$ is an additive ideal of $S$.
It follows that $S$ is $0$-simplifying if and only if $M_{\omega}(S)$ is $0$-simplifying.
\item $S$ is fundamental if and only if $M_{\omega}(S)$ is fundamental.
\item $S$ is semisimple if and only if $M_{\omega}(S)$ is semisimple.
\end{enumerate}
\end{proposition}
\begin{proof} (1) We set up an order-isomorphism between the additive ideals of $S$ and the additive ideals of  $M_{\omega}(S)$.

Let $I$ be an additive ideal of $S$.
Then $I$ is an inverse subsemigroup of $S$ and a Boolean inverse semigroup in its own right.
We therefore have that $M_{\omega}(I) \subseteq M_{\omega}(S)$.
Using Proposition~\ref{prop:ale}, it is easy to check that  $M_{\omega}(I)$ is an additive ideal of $M_{\omega}(S)$.

To prove the other direction, we shall need the following.
If $A$ is a finite square matrix then $A \oplus \mathbf{0}$ is the $\omega \times \omega$ 
rook matrix with entries from $A$ in the top left-hand side.
Let $P_{ij}(e)$ be the $\omega \times \omega$ matrix which is zero everywhere except in the $i$th row and $j$th column which contains the idempotent $e$;
this is a well-defined rook matrix.
Let $A$ be any $\omega \times \omega$ rook matrix with $s$ in the $k$th row and $l$th column.
Then the $\omega \times \omega$ rook matrix $P_{pk}(\mathbf{r}(s))AP_{lq}(\mathbf{d}(s))$
contains $s$ in the $p$th row and $q$th column and zeros elsewhere.

Let $\mathcal{I} \subseteq M_{\omega}(S)$ be an additive ideal of $M_{\omega}(S)$.
Let $J$ be the set of all elements of $S$ that appear in  some element of $\mathcal{I}$.
We prove first that $J$ is an additive ideal of $S$.
Let $a \in \mathcal{I}$ and let $s \in S$.
Then $a$ occurs as the $(i,j)$-entry of some rook matrix $A \in \mathcal{I}$.
Let $B$ be the rook matrix with zeros everywhere except that in the first row there is an $s$ in the $i$th-column.
Then $BA \in \mathcal{I}$ by assumption, but then $sa \in J$.
Similarly, we may show that $as \in J$.
It follows that $J$ is a semigroup ideal of $S$.
Let $a, b \in J$ such that $a \sim b$.
Let $a$ occur in the $\omega \times \omega$ matrix $A$.
By pre- and post-multiplying by suitable elements of $M_{\omega}(S)$
we can assume that $A$ contains exactly one non-zero entry.
Now using matrices of the form $P \oplus \mathbf{0}$,
where $P$ is a `permutation matrix', and by pre- and post-multiplying  
we may assume that $A$ has exactly one non-zero entry in the top left-hand corner.
We now see that $a \vee b \in J$.
We have proved that $J$ is an additive ideal of $S$.
We claim that $\mathcal{I} = M_{\omega}(J)$.
Let $A \in M_{\omega} (J)$.
Then we can write $A$ as a finite join of elements each of which has exactly one non-zero entry.
However, such matrices belong to  $\mathcal{I}$.
Thus $A \in \mathcal{I}$.
We have therefore proved that $M_{\omega}(J) \subseteq \mathcal{I}$.
The proof of the reverse inclusion is straightforward.
The above two processes are order-preserving and mutually inverse.
The result now follows.

(2) Let $S$ be fundamental.
We prove that $M_{\omega}(S)$ is fundamental.
Let $A \in M_{\omega}(S)$ commute with all idempotents.
We prove that $A$ is an idempotent.
We show first that all off-diagonal entries of $A$ are zero.
Suppose that $s \neq t$.
We prove that $a_{st} = 0$.
Let $E$ be the idempotent with $\mathbf{d}(a_{st})$ in the $t$th-row and $t$th-coloumn and zeros elsewhere.
Then by assumption $AE = EA$.
It follows that $a_{st} = 0$.
It follows that $A$ must be a diagonal matrix.
Consider the element $a_{ss}$.
We prove that it must be an idempotent.
Let $e$ be an arbitrary idempotent of $S$.
Let $E$ be the idempotent which has $e_{ss} = e$ and zeros elsewhere.
Then from $AE = EA$ we deduce that $a_{ss} e = e a_{ss}$.
But $S$ is fundamental and so $a_{ss}$ is an idempotent.
We have therefore proved that $M_{\omega}(S)$ is fundamental.
We now prove the converse.
Suppose that $M_{\omega}(S)$ is fundamental.
Let $a \in S$ be an element that commutes with all idempotents.
Let $A$ be the matrix with $a$ in the first row and first column.
Then $A$ commutes with all idempotents.
This implies $a$ is an idempotent and so $S$ is fundamental.

(3) The proof is immediate from the fact that a Boolean inverse semigroup is semisimple
if and only if $e^{\downarrow}$ is finite for each idempotent $e$ (by Lemma~\ref{lem:when-semisimple})
and the fact that the idempotents of $M_{\omega}(S)$ are the diagonal matrices with a finite number of non-zero idempotents on the diagonal (by Proposition~\ref{prop:ale}).
\end{proof}

A subset $F \subseteq \mathsf{E}(S)$ is said to be {\em self-conjugate} if $a^{-1}ea \in F$ for all $e \in F$ and $a \in S$.

\begin{proposition}\label{prop:order-isomorphisms} 
Let $S$ be a Boolean inverse semigroup.
There are order-iso\-morphisms between:
\begin{enumerate}

\item The poset of self-conjugate $\vee$-closed order ideals of $\mathsf{E}(S)$. 

\item The poset of additive ideals of $S$ 

\item The poset of all $o$-ideals of the type monoid of $S$.

\end{enumerate}
\end{proposition}
\begin{proof}  (1)$\Leftrightarrow$(2).
Let $A$ be an additive ideal of $S$.
Then $\mathsf{E}(A)$ is a self-conjugate $\vee$-closed order ideal of $\mathsf{E}(S)$.
Now, let $F$ be a self-conjugate $\vee$-closed order ideal of $\mathsf{E}(S)$.
Then $SFS$ is an additive ideal.
It is routine to check that $A = S \mathsf{E}(A)S$ and $F = \mathsf{E}(SFS)$.
Thus the functions $A \mapsto \mathsf{E}(A)$ and $F \mapsto SFS$ are mutually inverse.
The fact that both these functions are order-preserving is immediate from the definitions.

(2)$\Leftrightarrow$(3). This was proved in \cite[Proposition 4.2.4]{Wehrung}.
\end{proof}

The proof of the following is immediate from Proposition~\ref{prop:order-isomorphisms}.

\begin{corollary}\label{cor:rain} 
A Boolean inverse semigroup is $0$-simplifying if and only if its type monoid is simple.
\end{corollary}

\begin{example}
{\em $(\mathbb{N},+,)$ is a simple commutative idempotent conical refinement monoid.}
\end{example}

%%%%%%%%%%%%%%%%%%%%%%%%%%%%%%%%%%%%%%%%%%%%%%%%%%%%%%%%%%%%%%%%%%%%%%%%%%%
\begin{theorem}\label{them:type-semisimple} Let $S$ be a Boolean inverse semigroup.
Then the type monoid of $S$ is $\mathbb{N}$ if and only if $S$ is a $0$-simplifying semisimple Boolean inverse semigroup.
\end{theorem}
\begin{proof} Let $S$ be a Boolean inverse semigroup (which is not the zero semigroup)
and suppose that the type monoid of $S$ is $\mathbb{N}$.
Then $M_{\omega}(S)$ is $0$-simplifying by Corollary~\ref{cor:rain}.
Thus $S$ is $0$-simplifying by Proposition~\ref{prop:type-monoids}.
To show that $S$ is semisimple we shall use
the Dichotomy theorem, Theorem~\ref{them:dichotomy}. 
Thus we need to show that $S$ has an atom.
We have that $\tau (e) = 0$ if and only if $e = 0$.
We prove that if $\tau (e) = 1$ then $e$ is an atom.
Let $f \leq e$.
Then $e = f \oplus (e \setminus f)$.
Thus $\tau (e) = \tau (f) + \tau (e \setminus f)$.
It follows that $\tau (f) = 1$ or $\tau (f) = 0$.
If $\tau (f) = 0$ then $f = 0$.
If $\tau (f) = 1$ then $\tau (e \setminus f) = 0$ and so $ e \setminus f = 0$.
But $f \leq e$ and so $f = e$.
It remains to show that there are idempotents $e$ such that $\tau (e) = 1$.
By assumption, there is an idempotent $e$ such that $\tau (e) \neq 0$.
Suppose that $\tau (e) = m > 0$.
Then $\tau (e) = 1 + (m-1)$.
But the image of $\tau$ is a lower subset of the type monoid by \cite[Corollary 4.1.4]{Wehrung}.
Thus there is an idempotent $e'$ such that $\tau (e') = 1$.

Conversely, let $S$ be a $0$-simplifying semisimple Boolean inverse semigroup.
By Proposition~\ref{prop:type-fundamental}, we can assume that $S$ is fundamental.
Thus we may assume that $S$ is a simple semisimple Boolean inverse semigroup.
By part (3) of Corollary~\ref{cor:finite-stuff} and Corollary~\ref{cor:watch},
we have that $S \cong \mathscr{I}_{n}$ or $S \cong \mathscr{I}_{\rm \tiny fin}(X)$ where $X$ is an infinite set.
It is easy to compute the type monoids in both cases as $\mathbb{N}$.
\end{proof}

The following generalizes \cite[Proposition 4.1.9(1)]{Wehrung}.

\begin{lemma}\label{lem:cookies} Let $S_{i}$, where $i \in I$, be Boolean inverse semigroups.
Then the type monoid of the restricted direct product of the $S_{i}$ is equal to the restricted direct product of the type monoids of the $S_{i}$.
\end{lemma}
\begin{proof} We break this result down into a number of steps.\\

\noindent
(1) {\em Two nonzero idempotents $\mathbf{e}$ and $\mathbf{f}$ of $\bigoplus_{i \in I} S_{i}$
are $\mathscr{D}$-related if and only if they have the same support $J$ and $\mathbf{e}(i) \, \mathscr{D} \, \mathbf{f}(i)$
for all $i \in J$.} This result follows from the observation that if $a \mathscr{D} b$ in any inverse semigroup
then $a = 0$ if and only if $b = 0$.
In addition, $\mathbf{e} \, \mathscr{D} \, \mathbf{f}$ if and only if there exists an element $\mathbf{a}$ in 
$\bigoplus_{i \in I} S_{i}$ with support $J$ such that $\mathbf{e}(i) \stackrel{a_{i}}{\rightarrow} \mathbf{f}(i)$ for all
$i \in J$ where $\mathbf{a}(i) = a_{i}$.\\

\noindent
(2) {\em If $S_{i}$, where $i \in I$, are orthogonally separating Boolean inverse semigroups,
then $\bigoplus_{i \in I} S_{i}$ is orthogonally separating.}
Let $\mathbf{e}$ and $\mathbf{f}$ be two nonzero idempotents in $\bigoplus_{i \in I} S_{i}$.
Denote by $J \subseteq I$ the finite set on which both $\mathbf{e}$ and $\mathbf{f}$ are nonzero.
If $J$ is empty then $\mathbf{e}$ and $\mathbf{f}$ are already orthogonal.
Thus in what follows we may assume that $J$ is non-empty.
For each $j \in J$, let $\mathbf{e}(j)\stackrel{a_{j}}{\rightarrow} e_{j}'$
and $\mathbf{f}(j) \stackrel{b_{j}}{\rightarrow} f_{j}'$ be such that $e_{j}' \perp f_{j}'$;
this is possible since $S_{i}$ is orthogonally separating.
Define $\mathbf{e}'$  to be zero everywhere except that $\mathbf{e}'(j) = e_{j}'$, for $j \in J$,
and define $\mathbf{f}'$ to be zero everywhere except that  $\mathbf{f}'(j) = f_{j}'$, for $j \in J$.
Define $\mathbf{a}$ and $\mathbf{b}$ in the obvious ways.
We have that $\mathbf{e} \stackrel{\mathbf{a}}{\rightarrow} \mathbf{e}'$ and $\mathbf{f} \stackrel{\mathbf{b}}{\rightarrow} \mathbf{f}'$,
by part (1), above.
But $\mathbf{e}' \perp \mathbf{f}'$.
We have therefore shown that $\bigoplus_{i \in I} S_{i}$ is orthogonally separating.\\

\noindent
(3) {\em Suppose that $S_{i}$, where $i \in I$, are all orthogonally separating. 
Then 
$$\mathsf{T}(\bigoplus_{i \in I} S_{i}) \cong \bigoplus_{i \in I} \mathsf{T}(S_{i}).$$}
By part (2) above, we know that $\bigoplus_{i \in I} S_{i}$ is orthogonally separating.
Let $[\mathbf{e}] \in \mathsf{T}(\bigoplus_{i \in I} S_{i})$ be a nonzero idempotent with support $J$.
Define $F[\mathbf{e}]$ to be a function on $I$ which is $[0]$ everywhere except for $j \in J$
when it is $F[\mathbf{e}](j) = [\mathbf{e}(j)]$.
Observe that $J$ is also the support of $F[\mathbf{e}]$.
This is induces a well-defined map from $\mathsf{T}(\bigoplus_{i \in I} S_{i})$ to $\bigoplus_{i \in I} \mathsf{T}(S_{i})$
given by $[\mathbf{e}] \mapsto F[\mathbf{e}]$.
On the other hand, let $F$ be an element of $\bigoplus_{i \in I} \mathsf{T}(S_{i})$ with support $J'$.
Let $F(j) = [e_{j}]$ for $j \in J'$ and $[0]$ otherwise;
observe that there is choice for $e_{j}$
upto the $\mathscr{D}$-relation.
Define an idempotent $\mathbf{e}(j) = e_{j}$.
Observe that the support of $\mathbf{e}$ is also $J'$.
The idempotent $\mathbf{e}$ is determined upto the $\mathscr{D}$-relation.
It follows that the map $F \mapsto [\mathbf{e}]$, defined as above,  is in fact well-defined where we use part (1) above.
The above two maps are mutually inverse.
The map $[\mathbf{e}] \mapsto F[\mathbf{e}]$ induces an isomorphism of semigroups.\\

\noindent
(4) {\em We have that $M_{\omega} (\bigoplus_{i \in I} S_{i}) \cong \bigoplus_{i \in I} M_{\omega}(S_{i})$.}
Let $A \in M_{\omega} (\bigoplus_{i \in I} S_{i})$. 
Then for $p,q \in \omega$ we have that $A_{p,q}$ is an element of $\bigoplus_{i \in I} S_{i}$
and is a rook matrix.
Thus $A_{p,q}(i) \in S_{i}$.
If we fix $i$ and then let $p,q \in \omega$ vary, we get an $\omega \times \omega$ matrix $A^{(i)}$
such that $A^{(i)}_{p,q} = A_{p,q}(i)$.
It is immediate that $A^{(i)}$ is a rook matrix.
It follows that we have defined an element of $\bigoplus_{i \in I} M_{\omega}(S_{i})$.
We have therefore defined a function $A \mapsto (i \mapsto A^{(i)})$.
We now go in the other direction.
Let $F$ be an element of $\bigoplus_{i \in I} M_{\omega}(S_{i})$.
Then $F(i)$ is an element of $M_{\omega}(S_{i})$ and so is a rook matrix.
For $p,q \in \omega$, it follows that $F(i)_{p,q}$ is an element of $S_{i}$.
For fixed $p,q$, we now let $i$ vary. 
This gives us an element of 
$M_{\omega} (\bigoplus_{i \in I} S_{i})$.
The above two maps are mutually inverse.
The map $A \mapsto (i \mapsto A^{(i)})$ induces a semigroup isomorphism; recall that operations are pointwise.\\

We can now conclude our proof.
By definition, the type monoid of $\bigoplus_{i \in I} S_{i}$ is equal to
$\mathsf{T}(M_{\omega}(\bigoplus_{i \in I}S_{i}))$.
By step (4) above, this is equal to
$\mathsf{T}(\oplus_{i \in I} M_{\omega}(S_{i}))$.
By step (3) above this is isomorphic to $\bigoplus_{i \in I} \mathsf{T}(M_{\omega}(S_{i}))$.
By definition, this is the restricted direct product of the type monoids of the $S_{i}$.
\end{proof}

Recall that the free abelian monoid on the non-empty set $X$ is the free monoid on $X$ factored out by the congruence generated by 
$xy = yx$ for all $x,y \in X$.
We may regard the elements of the free abelian monoid on the set $X$ as being monomials of the form
$x_{1}^{m_{1}} \ldots x_{r}^{m_{r}}$ where $\{x_{1}, \ldots, a_{m}\} \subseteq X$.
It can also be regarded as the restricted direct product of copies of the natural numbers.
The elements $x \in X$ are atoms.

The following lemma is easy but worth stating explicitly.

\begin{lemma}\label{lem:ortho-join} We work in a Boolean inverse semigroup.
Let $e \, \mathscr{D} \,e_{1} \oplus e_{2}$.
Then there exist orthogonal idempotents $f_{1}$ and $f_{2}$ such that $e = f_{1} \oplus f_{2}$
where $e_{1} \, \mathscr{D} \, f_{1}$ and $e_{2} \, \mathscr{D} \, f_{2}$.
\end{lemma}
\begin{proof} Let $a \stackrel{a}{\rightarrow} e_{1} \oplus e_{2}$.
Then $\mathbf{r}(a) = e_{1} \oplus e_{2}$.
Put $a_{1} = e_{1}a$ and $a_{2} = e_{2}a$.
Then $a = a_{1} \oplus a_{2}$ where we have used Lemma~\ref{lem:buffs}.
Thus $e = \mathbf{d}(a) = \mathbf{d}(a_{1}) \oplus \mathbf{d}(a_{2})$ where 
$\mathbf{d}(a_{1}) \, \mathscr{D} \, e_{1}$
and 
$\mathbf{d}(a_{2}) \, \mathscr{D} \, e_{2}$, as required.
\end{proof}

\begin{theorem}\label{them:type-semisimple=general} The type monoid of a Boolean inverse semigroup $S$ is the free abelian monoid on
a non-empty set $X$ if and only if $S$ is semisimple.
\end{theorem}
\begin{proof} Suppose first that the type monoid of the Boolean inverse semigroup $S$ is the free abelian monoid on
a non-empty set $X$.
We prove that $T = M_{\omega}(S)$ is semisimple and use Proposition~\ref{prop:type-monoids} to deduce that $S$ is semisimple.
To prove that $T$ is semisimple,
we shall use Proposition~\ref{prop:atoms-semisimple}.
By assumption $\frac{\mathsf{E}(T)}{\mathscr{D}}$ is isomorphic to the free abelian monoid on the set $X$.
We prove a general result first.
Let $e$ be any non-zero idempotent of $T$.
Suppose that  $[e] = u + v$.
Since the type monoid is given by an isomorphism, 
there are idempotents $e_{1}$ and $e_{2}$ in $T$ such that $[e_{1}] = u$ and $[e_{2}] = v$. 
Thus $[e] = [e_{1}] + [e_{2}]$.
By assumption, $T$ is orthogonally separating and so 
$e \, \mathscr{D} \, e_{1}' \oplus e_{2}'$ where $e_{1} \, \mathscr{D} \, e_{1}'$ and $e_{2} \, \mathscr{D} \, e_{2}'$and $e_{1}' \perp e_{2}'$.
It follows by Lemma~\ref{lem:ortho-join}, that $e= f_{1} \oplus f_{2}$ where $[f_{1}] = u$ and $[f_{2}] = v$.
By induction, we deduce that if $[e] = u_{1} + \ldots + u_{s}$
then there are elements $e_{1},\ldots, e_{s}$ of $T$ which are pairwise orthogonal such that
$e = e_{1} \oplus \ldots \oplus e_{s}$ and $[e_{i}] = u_{i}$.
Observe that $[e] = 0$, the identity element of the free abelian monoid, if and only if $e = 0$.
The atoms in the free abelian monoid on $X$ are those elements $\mathbf{f}$ which assume the value zero for all $x \in X$
except that at $y$ we have that $\mathbf{f}(y) = 1$.
Let $e$ be an idempotent of $T$ such that $[e] = \mathbf{f}$, an atom as above.
Suppose that $f \leq e$.
Then $e = f \oplus (e \setminus f)$.
It follows that $[e] = [f] + [e \setminus f]$.
But $[e]$ is an atom of the free abelian monoid.
Thus either $[f] = 0$ and so $f = 0$ or $[e] = [f]$ and so $[e \setminus f] = 0$
from which it follows that $e \setminus f = 0$ and so $e = f$.
We have proved that the idempotents of $T$ that map to atoms of the free abelian monoid are themselves atoms.
We can now prove  that every non-zero element is above an atom.
Let $e$ be any non-zero idempotent.
If $e$ is an atom then we are done.
If $e$ is not an atom then there exists $f < e$.
But then $e = f \oplus (e \setminus f)$ and so $[e] = [f] + [e \setminus f]$.
If $[e \setminus f] = 0$ then $e \setminus f = 0$ and so $e = f$, which is a contradiction.
It follows that $[f] < [e]$.
Each non-zero element in the free abelian monoid is above an atom.
It follows that each non-zero element of $S$ is above an atom.
Next, we prove that each non-zero element is a finite join of atoms.
Let $e$ be a non-zero idempotent.
If $e$ is an atom we are done.
Otherwise $[e] = x_{1}^{m_{1}} \ldots x_{r}^{m_{r}}$ an element of the free abelian monoid.
The elements $x_{1},\ldots, x_{r}$ are atoms.
Let $[e_{i}] = x_{i}$ where $1 \leq i \leq r$ and $e_{i}$ is an atom.
Then $[e] = m_{1} [e_{1}] + \ldots + m_{r}[e_{r}]$.
We now apply our result above and deduce that $e$ is an orthogonal join of idempotents
each of which is $\mathscr{D}$-related to an atom.
By Lemma~\ref{lem:atom-idempotent}, we deduce that $e$ is an orthogonal join of a finite number of atoms.

To prove the converse, we need to show that the type monoid of a semisimple Boolean inverse semigroup $S$
is a free abelian monoid.
By Theorem~\ref{them:struct-semisimple} and Proposition~\ref{prop:semisimple-0-simplifying}
the semigroup $S$ is a restricted direct product of $0$-simplifying semisimple Boolean inverse semigroups.
By Theorem~\ref{them:type-semisimple}, the type monoid of a $0$-simplifying semisimple Boolean inverse semigroup
is $\mathbb{N}$. Thus the result will follow if we prove that the type monoid of a restricted direct product of Boolean inverse semigroups
is the restricted direct product of their type monoids.
This is proved as Lemma~\ref{lem:cookies}.
\end{proof}

%BIB
%%%%%%%%%%%%%%%%%%%%%%%%%%%%%%%%%%%%%%%%%%%%%%%%%%%%%%%%%%%%%%%%%%%%%%%%%%%%%%%%%%%%%%%%%%

\end{document}